\documentclass[12pt, a4paper,leqno]{amsart}
\usepackage{amsmath,amsthm,amscd,amssymb,amsfonts, amsbsy}
\usepackage{latexsym}
\usepackage{exscale}
\usepackage{mathrsfs}
\usepackage[utf8]{inputenc}

\usepackage[colorlinks,citecolor=red,pagebackref,hypertexnames=false, linkcolor=blue	, breaklinks]{hyperref}

\calclayout
\allowdisplaybreaks

\newcommand{\abs}[1]{\left|{#1}\right|}
\newcommand{\norm}[1]{\|{#1}\|}

\renewcommand{\emptyset}{\mbox{\textup{\O}}}

\newcommand{\br}[1]{{\left({#1}\right)}}
\newcommand{\bbr}[1]{{\left\{{#1}\right\}}}

\newcommand{\normH}[1]{{\vert\kern-0.25ex\vert\kern-0.25ex\vert #1 
		\vert\kern-0.25ex\vert\kern-0.25ex\vert}}

\newcommand{\normHB}[1]{{\Big\vert\kern-0.25ex\Big\vert\kern-0.25ex\Big\vert #1 
		\Big\vert\kern-0.25ex\Big\vert\kern-0.25ex\Big\vert}}

\def\Xint#1{\mathchoice
    {\XXint\displaystyle\textstyle{#1}}%
    {\XXint\textstyle\scriptstyle{#1}}%
    {\XXint\scriptstyle\scriptscriptstyle{#1}}%
    {\XXint\scriptscriptstyle\scriptscriptstyle{#1}}%
    \!\int}
    \def\XXint#1#2#3{{\setbox0=\hbox{$#1{#2#3}{\int}$}
    \vcenter{\hbox{$#2#3$}}\kern-.5\wd0}}
    \def\fint{\Xint-}
    
  \def\Xint#1{\mathchoice
    {\XXint\displaystyle\textstyle{#1}}%
    {\XXint\textstyle\scriptstyle{#1}}%
    {\XXint\scriptstyle\scriptscriptstyle{#1}}%
    {\XXint\scriptscriptstyle\scriptscriptstyle{#1}}%
    \!\int}
    \def\XXint#1#2#3{{\setbox0=\hbox{$#1{#2#3}{\int}$}
    \vcenter{\hbox{$#2#3$}}\kern-.5\wd0}}
    \def\dashint{\Xint-}

\def \A{ \mathbb{A} }

\def \L{ \mathbb{L} }
\def \N{ \mathbb{N} }
\def \R{ \mathbb{R} }

\def \Ncal { \mathcal{N} }
\def \Mcal { \mathcal{M} }

\def \Scal{ \mathcal{S} }
\def\Wcal{\mathcal{W} }

\def \hh{ \mathrm{H} }
\def \pp{ \mathrm{P} }
\def \Grm{ \mathrm{G} }

\newcommand{\wt}{\widetilde}
\newcommand{\wh}{\widehat}

\newcommand{\divv}{{\text{{\rm div}}}}
\renewcommand{\Re}{{\rm Re}\,}

\def\esssup{\mathop\mathrm{\,ess\,sup\,}}

\DeclareMathOperator{\supp}{supp}

\renewcommand{\chi}{{\bf 1}}

\theoremstyle{plain}
\newtheorem{theorem}[equation]{Theorem}
\newtheorem{lemma}[equation]{Lemma}

\newtheorem{proposition}[equation]{Proposition}

\newtheorem{claim}[equation]{Claim}

\theoremstyle{definition}
\newtheorem{definition}[equation]{Definition}

\theoremstyle{remark}
\newtheorem{remark}[equation]{Remark}

\numberwithin{equation}{section}

\begin{document}
\allowdisplaybreaks
\author{Li Chen}
\address{Li Chen
\\
Department of Mathematics
\\
Louisiana State University
\\
Baton Rouge, LA 70803-4918, USA}
\email{lichen@lsu.edu}

\author{José María Martell}
\address{Jos\'e Mar{\'\i}a Martell
\\
Instituto de Ciencias Matem\'aticas CSIC-UAM-UC3M-UCM
\\
Consejo Superior de Investigaciones Cient{\'\i}ficas
\\
C/ Nicol\'as Cabrera, 13-15
\\
E-28049 Madrid, Spain} \email{chema.martell@icmat.es}

\author{Cruz Prisuelos-Arribas}
\address{Cruz Prisuelos-Arribas
\\
Departamento de Física y Matemáticas	
\\
Universidad de Alcalá de Henares
\\
Plaza de San Diego, s/n
\\
E-28801 Alcalá de Henares, Madrid, Spain} 
\email{cruz.prisuelos@uah.es}

\title[The regularity problem for uniformly elliptic operators]{The regularity problem for uniformly elliptic operators in weighted spaces}

\thanks{The research leading to these results has received funding from the European Research Council under the European Union's Seventh Framework Programme (FP7/2007-2013)/ ERC	agreement no. 615112 HAPDEGMT. The authors also acknowledge financial support from the Spanish Ministry of Ministry of Science and Innovation, through the ``Severo Ochoa Programme for Centres of Excellence in R\&D'' (SEV-2015-0554). The second and third authors were partially supported by the Spanish Ministry of Science and Innovation,  MTM PID2019-107914GB-I00.}

\date{\today}


\begin{abstract}
This paper studies the regularity problem for block uniformly elliptic operators in divergence form with complex bounded measurable coefficients. We consider the case where the boundary data belongs to Lebesgue spaces with weights in the Muckenhoupt classes. Our results generalize those of S. Mayboroda (and those of P. Auscher and S. Stahlhut  employing the first order method) who considered the unweighted case. To obtain our main results we use the weighted Hardy space theory associated with elliptic operators recently developed by the last two named authors. One of the novel contributions of this paper is the use of an ``inhomogeneous'' vertical square function
which is shown to be controlled by the gradient of the function to which is applied in weighted Lebesgue spaces. 
\end{abstract}

\maketitle

\tableofcontents

\bigskip
\section{Introduction}
The study of elliptic boundary value problems on upper half spaces and on Lipschitz domains has a long history (see \cite{Ke94} for an introduction of major results for elliptic equations in divergence form with real symmetric coefficients). Recent breakthroughs in the field include \cite{HKMPII15, HKMPI15}, where the Dirichlet and  Regularity problems with real non-symmetric coefficients were considered.
 The study of elliptic problems with rough complex coefficients 
has been arousing great interest, 
 particularly after the solution of Kato's conjecture in \cite{AHLMT02}. In this direction, the connections between the Dirichlet, regularity and Neumann problems were studied in \cite{May10} for the block case. Some of the tools used in the latter reference include a Calderón-Zygmund theory adapted to singular ``non-integral" operators  (see \cite{Au07}, and also \cite{AMIII06, AMI07, AMII07}), as well as the a Hardy space theory adapted to elliptic operators (see \cite{HMay09, HMayMc11}). The reader is also referred to the work \cite{AS16} for the robust first order method to deal with more general cases. 

The main purpose of the present article is to  continue with this line of research and study the regularity problem for elliptic operators with block  structure, and with data in weighted Lebesgue spaces,  for Muckenhoupt weights. This is a natural problem to consider, in view of a well-established weighted Calderón-Zygmund theory in the series of papers \cite{AMIII06, AMI07, AMII07}  and the weighted Hardy space theory adapted to elliptic operators recently developed in \cite{MaPAI17, MaPAII17, PA17}. To state our main results, we need to introduce some background. Some  notation is taken from  \cite{Au07,MaPAI17,May10}. Let $A=(a_{jk})_{j,k=1}^n$ be an $n\times n$ matrix whose entries are $L^\infty$-valued complex coefficients defined on $\mathbb R^n$. We assume that $A$ satisfies the following uniform ellipticity condition: there exist $0<\lambda\le\Lambda<\infty$ such that
$$
\lambda\,|\xi|^2
\le
\Re A \,\xi\cdot\bar{\xi}
\quad\qquad\mbox{and}\qquad\quad
|A\,\xi\cdot \bar{\zeta}|
\le
\Lambda\,|\xi|\,|\zeta|,
$$
for all $\xi,\zeta\in\mathbb{C}^n$ and almost every $x\in \R^n$. We have used the notation
$\xi\cdot\zeta=\xi_1\,\zeta_1+\cdots+\xi_n\,\zeta_n$ and therefore
$\xi\cdot\bar{\zeta}$ is the usual inner product in $\mathbb{C}^n$, hence
$A(x)\,\xi\cdot\bar{\zeta}=\sum_{j,k}a_{j,k}(x)\,\xi_k\,\bar{\zeta_j}$. Associated with this matrix, consider the second order divergence form uniformly elliptic operator $
L:=-\divv (A\nabla)$. 

Given $A$ and $L$ as above we construct  the block matrix
$$
\A:= \br{\begin{array}{cc}A & 0 \\0 & 1\end{array}}. 
$$
Clearly, $\A$ is uniformly elliptic in $\R^{n+1}$ and, thus, $\A$ gives rise to the block divergence form uniformly elliptic operator 
\[
\L:=-\divv_{x,t} (\mathbb A\nabla_{x,t})
= -\divv_x (A\nabla_x)-\partial_{t}^2
=L_x-\partial_{t}^2.
\]
Here and elsewhere the points in $\R^{n+1}$ are written as $(x,t)\in\R^n\times \R$ so that $\nabla_{x,t}$ and $\divv_{x,t}$ denote respectively the full gradient and divergence, while $\nabla_x$ and $\divv_x$ stand respectively for the gradient and divergence in the first $n$ variables, and $L_x=L$. 

With the previous definition in hand, it is easy to see that whenever $f\in \mathcal{S}$, one has that $u(x,t):=e^{-t\sqrt L}f(x)$, where  $(x,t)\in\R_+^{n+1}$, is a weak solution of $\L u=0$ in $\R_{+}^{n+1}$. By this we mean that $u\in W_{\textrm{loc}}^{1,2}(\R_+^{n+1})$ satisfies 
\[
\iint_{\R_+^{n+1}} \A(x)\nabla_{x,t} u(x,t) \cdot \nabla_{x,t} \psi(x,t)\, dx\,dt =0,\qquad \forall\,\psi \in C_{0}^{\infty}(\R_+^{n+1}).
\]
Also, $u(\cdot,t)\to f$ in $L^2(\R^n)$ as $t\to 0^+$. We would like to observe that the latter convergence combined with \cite[Lemma~8.3]{AE} gives non-tangential convergence 
almost everywhere of the averages of $u$ over Whitney boxes to $f$ (see \cite[Remark~2.8]{AE} for more details).

Given $u\in L^2_{\textrm{loc}}(\R_+^{n+1})$ consider the $L^2$-non-tangential maximal function $\Ncal$ defined as
\begin{equation}
\Ncal u(x)
:=
\sup_{(y,t)\in \Gamma^\kappa(x)} \br{ \frac1{|D((y,t),\kappa t)|}\iint_{D((y,t),\kappa t)} |u(z,s)|^2 dzds}^{\frac12}, \,\,\,x\in \R^n,
\end{equation}
where $0<\kappa<1$ is a fixed small constant, $\Gamma^{\kappa}(x):=\{(y,t)\in \R_+^{n+1}:|x-y|< \kappa t\}$, and $D((y,t),\kappa t)$ denotes the $\R^{n+1}$-ball centered at $(y,t)$ with radius $\kappa t$, which is by construction contained in $\R^{n+1}_+$.

Using the previous definitions,  the regularity boundary value problem $(R_p)$, $1<p<\infty$, is said to be solvable for the operator $\L$ in $\R_{+}^{n+1}$ if, for every $f\in \mathcal{S}$, the weak solution to the equation $\L u=0$ (with boundary data $f$) given by the Poisson semigroup $u(x,t)=e^{-t\sqrt L}f(x)$, $(x,t)\in\R_+^{n+1}$, satisfies the non-tangential maximal function estimate
\begin{equation}\label{affr}
\norm{\Ncal(\nabla_{x,t}u)}_{L^p(\R^n)} \leq C\norm{\nabla f}_{L^p(\R^n)}
\end{equation}
and $u(\cdot,t)\to f$ in $L^2(\R^n)$ as $t\to 0^+$. As observed above, and since $f$ is a nice function, we already know that the latter convergence holds and hence one also has non-tangential convergence almost everywhere (see \cite{AE}). This means that in the solvability of the regularity problem what matters is the estimate \eqref{affr}, and for the convergence to boundary data there are several alternative conditions that one can take as part of the definition. 

The solvability of $(R_p)$, $1<p<\infty$, was proved by S. Mayboroda in \cite{May10}. In particular, the result is the following:
\begin{theorem}[{\cite[Thorem 4.1]{May10}}]\label{thm:May}
Let $\L$ be a block  elliptic operator in $\R_+^{n+1}$ as above.   Then for any $p$ so that $\max\bbr{1, \frac{n q_-(L)}{n+q_-(L)}}<p< q_+(L)$, the regularity problem $(R_p)$ is solvable. 
\end{theorem}
The interval $(q_-(L), q_+(L))$ represents  the maximal open interval on which the gradient of the heat semigroup $\{\sqrt t\nabla  e^{-tL}\}_{t>0}$ is uniformly bounded on $L^p(\R^n)$. More specifically:
\begin{align}
& q_-(L):=\inf\bbr{p\in (1,\infty): \sup_{t>0} \norm{\sqrt t\nabla  e^{-tL}}_{L^p(\R^n)\to L^p(\R^n)}< \infty},
\\
 & q_+(L):=\sup\bbr{p\in (1,\infty): \sup_{t>0} \norm{\sqrt t\nabla  e^{-tL}}_{L^p(\R^n)\to L^p(\R^n)}< \infty}.
\end{align}
Similarly,  $(p_-(L), p_+(L))$ denotes the maximal open interval of $p$ such that  the heat semigroup $\{e^{-tL}\}_{t>0}$ is  uniformly bounded on $L^p(\R^n)$. As we shall explain in more detail  in Section \ref{section:off}, these two intervals are somehow related, and for instance $p_-(L)=q_-(L)$, see \cite{Au07, AMII07} where these intervals were deeply studied.

\medskip

The goal of this paper is to extend Theorem \ref{thm:May} to the context of weighted Lebesgue spaces $L^p(w)$ with $w$ being a Muckenhoupt weight. 
To set the stage we first give some definitions and basic properties of Muckenhoupt weights.  For further details,
see~\cite{Du01, GCRF85, GrafakosI}. Given a weight $w$, that is,  a measurable function such that $0<w<\infty$ a.e. and $w\in L^1_{\rm loc}(\R^n)$, we say that $w\in A_p$, $1<p<\infty$, if
\[ 
[w]_{A_p} := \sup_B \left(\fint_B w(x)\,dx\right) \left(\fint_B
  w(x)^{1-p'}\,dx\right)^{p-1} < \infty, 
 \]
and, when $p=1$, we say that $w\in A_1$ if
\[ [w]_{A_1} := \sup_B \left(\fint_B w(x)\,dx\right)  \left(\esssup_{x\in B} w(x)^{-
	1}\right)<
\infty.  \]
Here and below the suprema are taken over all balls $B\subset \R^n$. 

We next introduce the reverse Hölder classes. We say that $w\in RH_s$, $1<s<\infty$, if
\[ [w]_{RH_s} := \sup_B \left(\fint_B w(x)\,dx\right )^{-1}
\left(\fint_B w(x)^s\,dx\right )^{1/s} < \infty, \]
and, that $w\in RH_\infty$, if
\[ [w]_{RH_\infty} := \sup_B\left(\fint_B w(x)\,dx\right)^{-1} \left(\esssup_{x\in B} w(x)\right)  <
\infty.  \]
Finally, we set
\[ A_\infty := \bigcup_{1\leq p <\infty} A_p  = \bigcup_{1<s\le \infty}
RH_s.  \]

We recall Muckenhoupt's theorem which states that $w\in A_p$, $1<p<\infty$, if and only if the Hardy-Littlewood maximal function 
$$
\mathcal{M} f(x):=\sup_{B\ni x} \fint_B|f(y)|\,dy,
$$
is bounded on $L^p(w)=L^p(w\,dx)$. In the case that $p=1$, the boundedness of $\mathcal{M}$ from $L^1(w)$ to $L^{1,\infty}(w)$ is equivalent to the fact that $w\in A_1$.

\medskip

It is well-known that $w\in A_\infty$ implies that  $w$ is a doubling measure. Indeed, for any $w\in A_r$, $1\leq r<\infty$, we have
\begin{align}\label{doublingcondition}
w(\lambda B)
\le
[w]_{A_r}\,\lambda^{n\,r}w(B),
\qquad \forall\,B,\ \forall\,\lambda>1.
\end{align}
Consequently,  $(\R^n,dw,|\cdot|)$ is a space of homogeneous type. 

Another important feature of Muckenhoupt classes is their openness. More precisely, 
the $A_p$ and $RH_s$ classes have the following self-improving
property: if $w\in A_p$, there exists $\epsilon>0$ such that $w\in
A_{p-\epsilon}$, and similarly if $w\in RH_s$, then $w\in
RH_{s+\delta}$ for some $\delta>0$.  These facts motivate the following definitions:
\begin{equation}
r_w:=\inf\big\{p:\ w\in A_p\big\}, \qquad s_w:=\inf\big\{q:\ w\in RH_{q'}\big\}.
\label{eq:defi:rw}
\end{equation}
Note that, according to our definition, $s_w$ is the conjugated exponent of the one defined in \cite[Lemma 4.1]{AMI07}.
Given $0\le p_0<q_0\le \infty$ and $w\in A_{\infty}$,  \cite[Lemma 4.1]{AMI07} implies that
\begin{align}\label{intervalrs}
\mathcal{W}_w(p_0,q_0):=\left\{p\in (p_0, q_0): \ w\in A_{\frac{p}{p_0}}\cap RH_{\left(\frac{q_0}{p}\right)'}\right\}
=
\left(p_0r_w,\frac{q_0}{s_w}\right).
\end{align}
If $p_0=0$ and $q_0<\infty$ it is understood that the only condition that stays is $w\in RH_{\left(\frac{q_0}{p}\right)'}$. Analogously, 
if $0<p_0$ and $q_0=\infty$ the only assumption is $w\in A_{\frac{p}{p_0}}$. Finally $\mathcal{W}_w(0,\infty)=(0,\infty)$.

\medskip

After these observations and definitions, we can introduce the notion of solvability in  the weighted context. More precisely, 
given a weight $w\in A_{\infty}$ we say that the weighted regularity boundary value problem $(R_p^w)$ is solvable for the operator $\L$ in $\R_{+}^{n+1}$, if, for every $f\in \mathcal{S}$,  the weak solution to the equation $\L u=0$ with the boundary data $f$ given by the Poisson semigroup $u(x,t)=e^{-t\sqrt L}f(x)$, $(x,t)\in\R_+^{n+1}$, satisfies the non-tangential maximal function estimate
\begin{equation}
\norm{\Ncal(\nabla_{x,t}u)}_{L^p(w)} \leq C\norm{\nabla f}_{L^p(w)}
\end{equation}
and 
$u(\cdot,t)\to f$ in $L^2(\R^n)$ as $t\to 0^+$.

Our main result establishes the solvability of $(R_p^w)$:
\begin{theorem}\label{thm:main}
Let $\L$ be a block  elliptic operator in $\R_+^{n+1}$ as above and let $w\in A_{\infty}$ be such that $\mathcal{W}_w(q_-(L),q_+(L))\neq \emptyset$. For every $p$ such that
\begin{equation}\label{range-p:thm}
\max\bbr{r_w, \tfrac{nr_w\widehat q_-(L)}{nr_w+\widehat q_-(L)}}<p<\frac{q_+(L)}{s_w},
\end{equation}
and for every $f\in \mathcal{S}$, if one sets $u(x,t)=e^{-t\sqrt L}f(x)$, $(x,t)\in\R_+^{n+1}$, then 
\begin{equation}\label{Est-main:thm}
\norm{\Ncal(\nabla_{x,t} u)}_{L^p(w)} \leq C\norm{\nabla f}_{L^p(w)}
\end{equation}
 and  $(R_p^w)$ is solvable.
\end{theorem}

In the previous result $(\wh q_-(L),\wh q_+(L))$ denotes the maximal open interval where $\{\sqrt{t} \nabla  e^{-t L}\}_{t>0}$ is uniformly bounded on $	L^p(w)$, or equivalently, where $\{\sqrt{t} \nabla  e^{-t L}\}_{t>0}$ satisfies  $L^p(w)$-$L^q(w)$ off-diagonal estimates on balls. 

We note that Theorem \ref{thm:main} is the natural extension of \cite[Thorem 4.1]{May10} (see Theorem \ref{thm:May} above). Indeed, if $w\equiv 1$ we have $r_w=1$ and $s_w=1$. Then, by definition, (see \cite{Au07} and Section \ref{section:off} below)
\[
(q_-(L),q_+(L))=(\wh q_-(L),\wh q_+(L))=\mathcal{W}_w(q_-(L),q_+(L)).
\]
Consequently, $\mathcal{W}_w(q_-(L),q_+(L))\neq \emptyset$, since $q_-(L)<2<q_+(L)$. Besides, 
\[
\max\bbr{r_w, \tfrac{nr_w\widehat q_-(L)}{nr_w+\widehat q_-(L)}}
=
\max\bbr{1, \tfrac{nq_-(L)}{n+ q_-(L)}},
\qquad
\frac{q_+(L)}{s_w}
=
q_+(L),
\]
and thus our range in Theorem \ref{thm:main} agrees with that of \cite[Thorem 4.1]{May10} (see Theorem \ref{thm:May} above).

Before explaining our approach to proving Theorem \ref{thm:main}, we  make an observation to clarify the hypothesis   requiring that the interval $\mathcal{W}_w(q_-(L),q_+(L))$ is not empty.
\begin{remark}
The fact that $\mathcal{W}_w(q_-(L),q_+(L))\neq \emptyset$ means that $\frac{q_+(L)}{q_-(L)}>r_w\, s_w$. This is a compatibility condition between $L$ and $w$. It  appears naturally in the theory developed in \cite{AMIII06, AMI07, AMII07} and forces the weight to be sufficiently good, depending on how close  $q_-(L)$ and $q_+(L)$ are to each other.
 As it was shown in \cite{Au07}, $q_-(L)=p_-(L)<\frac{2n}{n+2}$ while, in general, we only know that $q_+(L)>2$, actually, there are examples showing that $q_+(L)$ can be found arbitrarily close to two. Thus, 
$\frac{q_+(L)}{q_-(L)}>1+\frac2{n}$, and if $r_ws_w<1+\frac2{n}$, we always have that  $\mathcal{W}_w(q_-(L),q_+(L))\neq \emptyset$, for every $L$ as above. Of course, this is a very restrictive condition since we would expect to have wider classes of weights 
as the operator $L$ gets ``nicer''.
A very illustrative example is that of power weights of the form $\omega_\alpha(x):=|x|^{\alpha}$. The weight $\omega_\alpha$ belongs to $A_\infty$, if and only if, $\alpha>-n$. It is not hard to see that $r_w=\max\{1,1+\alpha/n\}$ and $s_w=\max\{1,(1+\alpha/n)^{-1}\}$ (see the definitions of $r_w$ and $s_w$ in \eqref{eq:defi:rw}) and hence $\frac{q_+(L)}{q_-(L)}>r_w\, s_w$ implies that $-n\big(1-\frac{q_-(L)}{q_+(L)}\big)<\alpha<n\big(\frac{q_+(L)}{q_-(L)}-1\big)$. Note that this range of $\alpha$ always contains the interval $[-\frac{2n}{n+2},2]$ and gets bigger as $q_-(L)$ decreases to one, and/or $q_+(L)$ increases to infinity.
\end{remark}

To proof Theorem \ref{thm:main} we independently  consider the estimate of the weighted norm of the non-tangential maximal function for the time derivative $\partial_t e^{-t\sqrt L}f$, and the one for the spatial derivatives $\nabla e^{-t\sqrt L}f$. 
For the former we use ideas from \cite[Theorem 4.1]{May10} combined with the weighted Hardy space theory associated with elliptic operators developed in  \cite{MaPAI17,MaPAII17,PA17}.\footnote{We note that there is an alternative and independent method developed in the forthcoming paper \cite{ChMPA18} to deal with the regularity problem for degenerate elliptic operators which is based on some adapted Calderón-Zygmund theory, and does not require to deal with the highly technicality of Hardy spaces. The corresponding result in \cite{ChMPA18} gives however a more restricted range.}

Regarding, the spatial derivatives we follow a different path, for which we need to introduce the following ``inhomogeneous'' vertical square function 
\begin{equation}\label{eq:wtG}
\wt \Grm_{\mathrm H} f(x):=\br{\int_{0}^{\infty} \abs{t^2 \nabla L e^{-t^2 L}f(x)}^2 \frac{dt}{t}}^{\frac12}.
\end{equation}
Note that compared with the usual vertical square functions (or Littlewood-Paley-Stein functionals), the power of $t$ does not seem to be consistent with the order of the operator in the integral. Indeed, the power of $t$ corresponding to the usual square function would be three instead of two. That is why we use the terminology ``inhomogeneous''.  This has the effect of getting an
estimate  for $\wt \Grm_{\mathrm H} f$ in terms of $\nabla f$ in $L^q(w)$, rather than obtaining   an estimate from $L^q(w)$ to $L^q(w)$ (which is the natural one for the usual square functions).  We would like to mention that in the case that $L$ is the Laplace operator, this inhomogeneous square function (as well as its relation to reverse Riesz transform estimates), was studied in \cite{CD03} in the setting of Riemannian manifolds.
 
 As we shall see below, the estimate regarding the non-tangential maximal function for the spatial derivatives can be inferred from the following estimate of the inhomogeneous vertical square function introduced in \eqref{eq:wtG}. This approach is one of the novel contributions of this paper even in the unweighted case. 
\begin{theorem}\label{thm:wtG}
Let $w\in A_{\infty}$ and assume that  $\mathcal{W}_w(q_-(L),q_+(L))\neq \emptyset$.  Given $q$ so that $\max\bbr{r_w, \frac{nr_w\wh q_-(L)}{nr_w+\wh q_-(L)}}<q <\wh q_+(L)$ there holds
\begin{equation}\label{eq:Lp-wtG}
\norm{\wt \Grm_{\mathrm H} f}_{L^q(w)} \lesssim \norm{\nabla f}_{L^q(w)}, 
\end{equation}
for every $f\in \mathcal{S}$.
\end{theorem}

\medskip

The plan of the paper is as follows. In Section \ref{section:off} we present some preliminaries and some auxiliary results needed in Sections \ref{sect:inh-SF} and  \ref{section:main-proof}. As for these two last sections, the former is devoted to the proof of Theorem \ref{thm:wtG}, and the later to proving Theorem \ref{thm:main}.

%

%
%
%
%
%
\section{Preliminaries and Auxiliary results}\label{section:off}
Throughout the paper the letters $c$, $C$, $\theta$ or $\widetilde{\theta}$ will represent any harmless constant that will not depend on any relevant parameter of the corresponding computation, and that will may be different  in both sides of an inequality. Moreover, 
we shall write $\nabla=(\partial_{x_1},\dots\partial_{x_n})$ to denote the gradient in $\R^n$. We will typically write the elements in  $\R^{n+1}_+$ as $(x,t)$ with $x\in\R^n$ and $t>0$ and hence $\nabla_{x,t}=(\nabla,\partial t)$ stands for the gradient in $\R^{n+1}$.
Given a ball $B$, we use the notation $C_1(B)=4B$ and  $C_j(B)=2^{j+1}B\backslash 2^jB$ for $j\geq 2$. If $w\in A_{\infty}$, we write
$$
\fint_B hdw=\frac{1}{w(B)}\int_B hdw,\qquad \fint_{C_j(B)} hdw=\frac{1}{w(2^{j+1}(B))} \int_{C_j(B)} hdw.
$$

We recall the notion of full off-diagonal estimates:
\begin{definition}
Let $\{T_t\}_{t>0}$ be a family of sublinear operators and let $1\leq p\leq q\leq \infty$. We say that $\{T_t\}_{t>0}$ satisfies 
$L^p$-$L^q$ full off-diagonal estimates, denoted by $T_t\in \mathcal F(L^p(\R^n)-L^q(\R^n))$, if there exists a constant $c>0$ such that 
for all closed sets $E$ and $F$, all $f$ and all $t>0$ we have 
\begin{equation}\label{off:full}
\|T_t(\chi_E f)\chi_F\|_{L^q(\R^n)}
\lesssim 
t^{-\frac{1}{2}(\frac{n}{p}-\frac{n}{q})} e^{-\frac{cd(E,F)^2}{t}} \|f\chi_E\|_{L^p(\R^n)},
\end{equation}
where $d(E,F)=\inf\{|x-y|:x\in E, y\in F\}$.
\end{definition}
\begin{definition}
Given $1\leq p\leq q\leq \infty$ and any weight $w\in A_{\infty}$. We say that a family of sublinear operators $\{T_t\}_{t>0}$ satisfies 
$L^p$-$L^q$ off-diagonal estimates on balls, denoted by $T_t\in \mathcal O(L^p(w)-L^q(w))$, if there exist $\theta_1,\theta_2>0$ and $c>0$ such that for all $t>0$ and for any ball $B$ with radius $r_B$, 
\begin{equation}\label{off:BtoB}
\br{\fint_B \abs{T_t(f \chi_B)}^q dw}^{\frac1q} \lesssim \Upsilon\br{\frac{r_B}{\sqrt t}}^{\theta_2} \br{\fint_B |f|^p dw}^{\frac1p},
\end{equation}
and, for all $j\geq 2$,
\begin{equation}\label{off:CtoB}
\br{\fint_B \abs{T_t(f \chi_{C_j(B)})}^q dw}^{\frac1q} 
\lesssim 2^{j\theta_1} \Upsilon\br{\frac{2^j r_B}{\sqrt t}}^{\theta_2} e^{-\frac{c4^j r_B^2}{t}} \br{\fint_{C_j(B)} |f|^p dw}^{\frac1p},
\end{equation}
and
\begin{equation}\label{off:BtoC}
\br{\fint_{C_j(B)} \abs{T_t(f \chi_B)}^q dw}^{\frac1q} 
\lesssim 2^{j\theta_1} \Upsilon\br{\frac{2^j r_B}{\sqrt t}}^{\theta_2} e^{-\frac{c4^j r_B^2}{t}} \br{\fint_B |f|^p dw}^{\frac1p}.
\end{equation}
\end{definition}
In the previous definition $\Upsilon(s)$ is defined as $\Upsilon(s):=\max\{s,s^{-1}\}$ for $s>0$.
Besides note that if $q=\infty$ we would consider the corresponding $L^{\infty}(\R^n)=L^\infty(w)$ norms, in place of the $L^q(w)$-norms.

For the following results we recall that $(p_-(L), p_+(L))$  (resp.~$(q_-(L), q_+(L))$) denotes  the maximal open interval on which $\{e^{-tL}\}_{t>0}$ (resp. $\{\sqrt t \nabla e^{-tL}\}_{t>0}$) is uniformly bounded on $L^p(\R^n)$.

\begin{proposition}[{\cite[Chapter 3]{Au07}}]\label{prop:full-off}
	\ 
	
\begin{list}{$(\theenumi)$}{\usecounter{enumi}\leftmargin=.8cm \labelwidth=.8cm \itemsep=0.1cm \topsep=.2cm \renewcommand{\theenumi}{\roman{enumi}}}

\item If $p_-(L)<p\leq q<p_+(L)$, then $\{(tL)^m e^{-tL}\}_{t>0} \in \mathcal F(L^p-L^q)$ for every $m\in \mathbb N$.

\item If $q_-(L)<p\leq q<q_+(L)$,  then $\{\sqrt t \nabla (tL)^m e^{-tL}\}_{t>0} \in \mathcal F(L^p-L^q)$ for every $m\in \mathbb N$. In addition, we have $p_-(L)=q_-(L)$ and $(q_-(L))^{*} \leq p_+(L)$. 

\item $p_-(L)$, $q_+(L)\le (q_+(L))^*\le p_+(L)$, and $q_+(L)>2$.

\item $p_-(L)=1$ and $p_+(L)=\infty$ if $n=1, 2$; $q_+(L)=\infty$ if $n=1$; and $p_-(L)<\frac{2n}{n+2}$, $p_+(L)>\frac{2n}{n-2}$ if $n\ge 3$.

\end{list}
\end{proposition}

\begin{proposition}[{\cite[Proposition 5.9]{AMII07}}] \label{prop:weightedOD}
 Let $w\in A_{\infty}$.

\begin{list}{$(\theenumi)$}{\usecounter{enumi}\leftmargin=.8cm \labelwidth=.8cm \itemsep=0.1cm \topsep=.2cm \renewcommand{\theenumi}{\roman{enumi}}}
\item Assume $\mathcal{W}_w(p_-(L),p_+(L))\neq \emptyset$. There exists a maximal open interval denoted by $(\wh p_-(L),\wh p_+(L))$ containing 
$\mathcal{W}_w(p_-(L),p_+(L))$ such that if $\wh p_-(L)<p\le q<\wh p_+(L)$, then  $\{(tL)^m e^{-tL}\}_{t>0} \in \mathcal O(L^p(w)-L^q(w))$ and is a bounded set in $\mathcal L(L^p(w))$.

\item Assume $\mathcal{W}_w(q_-(L),q_+(L))\neq \emptyset$. There exists a maximal open interval denoted by $(\wh q_-(L),\wh q_+(L))$ containing 
$\mathcal{W}_w(q_-(L),q_+(L))$  such that if $\wh q_-(L)<p\le q<\wh q_+(L)$, then $\{\sqrt t\nabla (tL)^m e^{-tL}\}_{t>0} \in \mathcal O(L^p(w)-L^q(w))$ and is a bounded set in $\mathcal L(L^p(w))$.

\item $\wh p_-(L)=\wh q_-(L)$ and $(\wh q_+(L))_w^{*} \leq \wh p_+(L)$. 
\end{list}
\end{proposition} 
In the above proposition $({q}_+(L))^{*} $ and  $(\widehat{q}_+(L))_w^{*} $ are defined as follows: for all $0<q<\infty$, 
\begin{align}\label{q_wstar}
(q)_w^*:=\begin{cases}\frac{qnr_w}{nr_w-q}, & nr_w>q,
\\
\infty, & \textrm{otherwise},
\end{cases}
\quad \textrm{and}\quad
(q)^*:=\begin{cases}\frac{qn}{n-q}, & n>q,
\\
\infty, & \textrm{otherwise}.
\end{cases}
\end{align}


The following result, which appears in a more general way in \cite[Proposition 4.1]{PA17}, and it is a weighted version of \cite[(5.12)]{MaPAII17} (see also \cite{HMay09}), contains some off-diagonal estimates for the family $\{\mathcal{T}_{t,s}\}_{s,t>0}:=\{(e^{-t^2L}-e^{-(t^2+s^2)L})^M\}_{s,t>0}$, where $M\in \N$.

\begin{proposition}\label{prop:lebesgueoff-dBQ}
Let $w\in A_{\infty}$, $\mathcal{W}_w(p_-(L),p_+(L))\neq \emptyset$, $p\in(\widehat{p}_-(L),\widehat{p}_+(L))$, and let $0<t,s<\infty$. Given $M\in \N$, for all sets $E_1,E_2\subset \R^n$ and $f\in L^p(w)$ such that $\supp (f)\subset  E_1$,  we have that $\{\mathcal{T}_{t,s}\}_{s,t>0}:=\{(e^{-t^2L_w}-e^{-(t^2+s^2)L_w})^M\}_{s,t>0}$ satisfies the following $L^p(w)$-$L^p(w)$ off-diagonal estimates: 
\begin{align}\label{AB}
\left\|\chi_{E_2}\mathcal{T}_{t,s}f\right\|_{L^{p}(w)}
\lesssim
\left(\frac{s^2}{t^2}\right)^M
e^{-c\frac{d({E}_1,{E}_2)^2}{t^2+s^2}}\|f\chi_{E_1}\|_{L^{p}(w)}.
\end{align}
In particular, 
\begin{equation}\label{boundednesstsr}
\|\mathcal{T}_{t,s}f\|_{L^{p}(w)}\lesssim \left(\frac{s^2}{t^2}\right)^M\|f\|_{L^{p}(w)}.
\end{equation}
\end{proposition}


We introduced before the Hardy-Littlewood maximal function and here we present some weighted maximal version which will be used throughout the paper. Given $w\in A_\infty$ and  $0<q<\infty$ we set 
\begin{align}\label{weightedmaximal}
\mathcal{M}_q^w f(x):=\sup_{B\ni x} \left(\fint_B|f(y)|^q\,dw\right)^{\frac{1}{q}}, \quad 1\leq q<\infty.
\end{align}
Since  $w\in A_\infty$ implies that $w$ is a doubling measure (see \eqref{doublingcondition}) then $\mathcal{M}_q^w $ is bounded on $L^p(w)$ for every $q<p\le \infty$ and is bounded from $L^q(w)$ to $L^{q,\infty}(w)$.





The following result contains a Calderón-Zygmund decomposition where the function is split according to the level sets of its gradient and to its norm in weighted Lebesgue spaces.

\begin{lemma}[{\cite[Lemma 6.6]{AMIII06}, \cite[Proposition 9.1]{AMI07}}]\label{lemma:C-Z-decomposition}
Let $n\geq 1$, $w\in A_{\infty}$, $\mu:=wdx$, and $r_w<p_0<\infty$ (with the possibility of taking $p_0=1$ if $r_w=1$). Assume that $h\in \mathcal{S}$, and let $\alpha>0$. Then, one can find a collection of balls $\{B_i\}_{i\in \N}$ (with radii $r_{B_i}$), smooth functions $b_i$, and a function $g\in L^1_{loc}(w)$ such that
$$
h=g+\sum_{i\in \N}b_i,\qquad \text{in $L^{p_0}(w)$},
$$
and the following properties hold
\begin{align}\label{C-Z:g}
|\nabla g(x)|\leq C\alpha, \textrm{ for } \mu\textrm{-a.e. } x,
\end{align}

\begin{align}\label{C-Z:b}
\supp b_i\subset B_i\quad \textrm{and}\quad \int_{B_i}|\nabla b_i(x)|^{p_0}w(x)dx\leq C\alpha^{p_0}w(B_i),
\end{align}

\begin{align}\label{C-Z:sum}
\sum_{i\in \N} w(B_i)\leq \frac{C}{\alpha^{p_0}}\int_{\R^n}|\nabla h(x)|^{p_0}w(x)dx,
\end{align}

\begin{align}\label{C-Z:sumoverlap}
\sum_{i\in \N}\chi_{B_i}\leq N,
\end{align}
where $C$ and $N$ depend only on the dimension, the doubling constant of $\mu$, and $p_0$. In addition, for $1\leq q<(p_0)_w^*$, where $(p_0)_w^*$ is defined in \eqref{q_wstar}, we have
\begin{align}\label{C-Z:extrab}
\left(\dashint_{B_i}|b_i(x)|^qdw\right)^{\frac{1}{q}}\lesssim \alpha r_{B_i}.
\end{align}
\end{lemma}

The estimates contained in the following auxiliary result follows easily from Hölder's inequality along with the definitions of the $A_p$ and $RH_s$ classes: 
\begin{lemma}\label{lemma:sin-con}
For every  $0<p\leq q<\infty$, every ball $B\subset\R^n$ and  every $j\geq 1$, there hold
\[
\bigg(\dashint_{C_j(B)}|f(x)|^{p}dx\bigg)^{\frac{1}{p}}
\lesssim
\bigg(\dashint_{C_j(B)}|f(x)|^{q}dw\bigg)^{\frac{1}{q}},\qquad \forall\,w\in A_{\frac{q}{p}},
\]
and 
\[
\bigg(\dashint_{C_j(B)}|f(x)|^{p}dw\bigg)^{\frac{1}{p}}
\lesssim
\bigg(\dashint_{C_j(B)}|f(x)|^{q}dx\bigg)^{\frac{1}{q}},
\qquad
\forall\, w\in RH_{\left(\frac{q}{p}\right)'},
\]
where the implicit constants are independent of $j$.
\end{lemma}

We also need the following  off-diagonal estimate on Sobolev spaces, which can be proved as  \cite[(4.6)]{Au07}. Nevertheless, we include the proof for the sake of completeness.


\begin{proposition}\label{prop:Gsum}
Let $w\in A_{\infty}$. Given $p$, $q$ with $\max\{r_w, \widehat{q}_-(L)\}<p\le q<\widehat{q}_+(L)$, $x\in \R^n$, $s>0$, $\lambda\ge 1$, and 
$H$ defined in $\R^{n+1}_+$ there hold
\[
\bigg(\fint_{B(x,\lambda s)} |\nabla e^{-s^2L}H(\cdot\,,s)|^q dw\bigg)^{\frac1q} 
\lesssim 
e^{-c\lambda^2}\sum_{j=1}^{\infty} e^{-c4^j}  \bigg(\fint_{B(x,2^{j+1}\lambda s)} |\nabla H(\cdot\,,s)|^{p} dw\bigg)^{\frac{1}{p}},
\]
whenever $\nabla H(\cdot,s)\in L^p(w)$
and
\[
\bigg(\fint_{B(x,\lambda s)} |\nabla e^{-s^2L}H(\cdot\,,s)|^q dw\bigg)^{\frac1q} 
\lesssim 
e^{-c\lambda^2}\sum_{j=1}^{\infty} e^{-c4^j}  \bigg(\fint_{B(x,2^{j+1}\lambda s)} |\nabla e^{-\frac{s^2}{2}L}H(\cdot\,,s)|^{p} dw\bigg)^{\frac{1}{p}},
\]
whenever $H(\cdot,s)\in L^p(w)$.
\end{proposition}
\begin{proof}
It is straightforward to see that it suffices to prove the first estimate. Indeed assuming that, we can write $t=s/\sqrt{2}$ and $\widetilde{H}(\cdot,t)=e^{-t^2 L} H(\cdot,\sqrt{2}t)$ to easily get 
\begin{align*}
&\bigg(\fint_{B(x,\lambda s)} |\nabla e^{-s^2L}H(\cdot\,,s)|^q dw\bigg)^{\frac1q}
=
\bigg(\fint_{B(x,\lambda \sqrt{2} t)} |\nabla e^{-t^2L}\widetilde{H}(\cdot\,,t)|^q dw\bigg)^{\frac1q} 
\\
&\qquad\qquad\lesssim
e^{-2\,c\lambda^2}\sum_{j=1}^{\infty} e^{-c4^j}  \bigg(\fint_{B(x,2^{j+1}\lambda\sqrt{2} t)} |\nabla \widetilde{H}(\cdot\,,t)|^{p} dw\bigg)^{\frac{1}{p}}
\\
&\qquad\qquad=
e^{-2\,c\lambda^2}\sum_{j=1}^{\infty} e^{-c4^j}  \bigg(\fint_{B(x,2^{j+1}\lambda s)} |\nabla e^{-\frac{s^2}{2}L}H(\cdot\,,s)|^{p} dw\bigg)^{\frac{1}{p}}.
\end{align*}

In order to prove the first estimate, fix $(x,s)\in \R^{n+1}_+$ and define $B:=B(x,\lambda s)$. Writing $H_s(\cdot):=H(\cdot\,,s)$, by the conservation property ($e^{-sL}1=1$, $s>0$) we have that
$$
\nabla e^{-s^2L}H(\cdot\,,s)=\nabla  e^{-s^2L}(H_s-(H_s)_{{4B}}),
$$
where $(H_s)_{{4B}}:=\dashint_{4B}H(\cdot\,,s)dw$.
Consequently, applying the $L^p(w)$-$L^{q}(w)$ off-diagonal estimates on balls satisfied by  $\sqrt{\tau} \nabla e^{-\tau L}$,
\begin{align*}
&\bigg(\fint_{B(x,\lambda s)} |\nabla e^{-s^2L}H(\cdot\,,s)|^q dw\bigg)^{\frac1q} 
\\
&\quad\lesssim 
\sum_{j\geq 1}2^{j(\theta_2+\theta_1)}\lambda^{\theta_2}\frac{e^{-c4^j\lambda^2}}{s}
\bigg(\fint_{B(x,2^{j+1}\lambda s)} |H_s-(H_s)_{{4B}}|^p dw\bigg)^{\frac1p} 
\\
&\quad
\lesssim e^{-c\lambda^2}
\sum_{j\geq 1}\frac{e^{-c4^j}}{s}\left(
\bigg(\fint_{2^{j+1}B} |H_s-(H_s)_{{2^{j+1}B}}|^p dw\bigg)^{\frac1p}+\sum_{l=2}^j\left|(H_s)_{{2^{l+1}B}}-(H_s)_{{2^{l}B}}\right| \right)
\\
&\quad
\lesssim e^{-c\lambda^2}
\sum_{j\geq 1}\frac{e^{-c4^j}}{s}\sum_{l=2}^j
\bigg(\fint_{2^{l+1}B} |H_s-(H_s)_{{2^{l+1}B}}|^p dw\bigg)^{\frac1p}.
\end{align*}
Hence, by Poincar\'e inequality, which holds since $p>r_w$ and thus $w\in A_p$, we conclude that
\begin{align*}
\bigg(\fint_{B(x,\lambda s)} |\nabla e^{-s^2L}H(\cdot\,,s)|^q dw\bigg)^{\frac1q}  &
\lesssim e^{-c\lambda^2}
\sum_{j\geq 1}\frac{e^{-c4^j}}{s}\sum_{l=2}^j
\bigg(\fint_{2^{l+1}B} |H_s-(H_s)_{{2^{l+1}B}}|^p dw\bigg)^{\frac1p}
\\&
\lesssim e^{-c\lambda^2}
\sum_{j\geq 1}e^{-c4^j}\sum_{l=2}^j2^{l+1}
\bigg(\fint_{2^{l+1}B} |\nabla H_s|^p dw\bigg)^{\frac1p}
\\&
\lesssim e^{-c\lambda^2}
\sum_{l\geq 1}e^{-c4^l}
\bigg(\fint_{2^{l+1}B} |\nabla H_s|^p dw\bigg)^{\frac1p}.
\end{align*}
This finishes the proof.
\end{proof}


\section{Proof of Theorem \ref{thm:wtG}}\label{sect:inh-SF}

In this section, we shall prove Theorem \ref{thm:wtG}, which establishes weighted norm estimates for $\wt \Grm_{\mathrm H}$ (defined in \eqref{eq:wtG}). 
To this end, we introduce two results that will be used in that proof.
\begin{theorem}[{\cite[Theorem 6.2]{AMIII06}}]\label{thm:wRR}
Let $w\in A_{\infty}$ be such that $\mathcal{W}_w(p_-(L),p_+(L))\neq \emptyset$. Given $p$,  $\max\bbr{r_w, \tfrac{nr_w \widehat p_-(L)}{nr_w+\widehat p_-(L)}}<p<\widehat{p}_+(L)$, there holds
\begin{equation}\label{RRp}
\norm{L^{\frac12}f}_{L^p(w)} \lesssim \norm{\nabla f}_{L^p(w)},
\end{equation}
for every $f\in \mathcal{S}$. 
\end{theorem}

Note that recalling \eqref{intervalrs}, and using Propositions \ref{prop:full-off} and \ref{prop:weightedOD}, it is not difficult to see that the range where  \eqref{RRp} holds contains $\Wcal_w(p_-(L),p_+(L))$ and hence also $\Wcal_w(q_-(L),q_+(L))$.

The following result deals with some vertical square functions. It can be proved with an argument similar to that of \cite[Theorem 7.2]{AMIII06}, or also combining \cite[Proof of Proposition 10.1]{CMR15} and \cite[Theorem 1.12]{MaPAI17}). Further details are left to the interested reader. 
\begin{lemma}\label{lem:wVSF}
Let $w\in A_\infty$ be such that $\mathcal{W}_w(q_-(L),q_+(L))\neq \emptyset$. If $\wh q_-(L)< p< \wh q_+(L)$ then  for all $f\in L^\infty_c(\R^n)$ we have
\begin{align*}
\Bigg\|\br{\int_{0}^{\infty} |t^2 \nabla \sqrt{L} e^{-t^2 L}f|^2 \frac{dt}{t}}^{\frac12}\Bigg\|_{L^p(w)} +\Bigg\|\br{\int_{0}^{\infty} |t^3 \nabla  L e^{-t^2 L}f|^2 \frac{dt}{t}}^{\frac12}\Bigg\|_{L^p(w)}
\lesssim 
\|f\|_{L^p(w)}.
\end{align*}
\end{lemma}


We next define some auxiliary square functions that will be also useful in the proof of Theorem \ref{thm:wtG}. In particular, consider
\begin{align*}
\mathrm{G}_{\hh}f=\left(\int_0^{\infty}\int_{B(x,t)}|t\nabla t\sqrt{L}e^{-t^2L}f(y)|^2\frac{dy\,dt}{t^{n+1}}\right)^{\frac{1}{2}}
\end{align*}
and, for $m=1,2$,
\begin{align}\label{conicalheat}
{\mathcal{S}}_{m,\hh}f=\left(\int_0^{\infty}\int_{B(x,t)}|(t\sqrt{L})^me^{-t^2L}f(y)|^2\frac{dy\,dt}{t^{n+1}}\right)^{\frac{1}{2}}.
\end{align}


\subsection{Proof of Theorem \ref{thm:wtG} }
First of all fix $f\in \mathcal{S}$ and consider $T_{t}:=\nabla e^{-tL}$ and $F(y,t):=t^2Lf(y)$, for $(y,t)\in \R^{n+1}_+$. Then, fix $(x,t)\in \R^{n+1}_+$, and note that for $B:=B(x,t)$ and by the second estimate in Proposition \ref{prop:Gsum} with $w=1$ and $\lambda=1$, we have that, for all $2<q_0<q_+(L)$,
\begin{align*}
\left(\dashint_{B(x,t)}|T_{t^2}F(y,t)|^{q_0} dy\right)^{\frac{1}{q_0}}
\lesssim
\sum_{j\geq 1}e^{-c4^j}
\left(\dashint_{B(x,2^{j+1}t)}|T_{t^2/2}F(y,t)|^{2} dy\right)^{\frac{1}{2}}.
\end{align*}
Besides note that for any constant $c>0$, $F(y,ct)=c^2F(y,t)$. Hence, we can apply \cite[Proposition 4.2, (b)]{PA18} 
 to obtain that, for all $0<q<\frac{q_+(L)}{s_w}$,
\begin{align*}
\big\|\wt \Grm_{\mathrm H}f\big\|_{L^q(w)}
\lesssim \big\|\mathrm{G}_{\hh}\big(\sqrt{L}f\big)\big\|_{L^q(w)}.
\end{align*}
Moreover, by \cite[Proposition 4.16, (a)]{PA18} and  \cite[Proposition 4.5, (b)]{PA18}, we have that, for all $0<q<\frac{(p_+(L))^*}{s_w}$
\begin{align}\label{usedfinal1}
\big\|\mathrm{G}_{\hh}\big(\sqrt{L}f\big)\big\|_{L^q(w)}
\lesssim  \big\|{\Scal}_{1,\hh}\big(\sqrt{L}f\big)\big\|_{L^q(w)}
\lesssim  \|{\Scal}_{2,\hh}\big(\sqrt{L}f\big)\big\|_{L^q(w)}.
\end{align}
Thus, since $\frac{q_+(L)}{s_w}\leq\frac{(p_+(L))^*}{s_w}$, we obtain that, for all $0<q<\frac{q_+(L)}{s_w}$,
\begin{align}\label{2}
\|\wt \Grm_{\mathrm H}f\|_{L^q(w)}
\lesssim 
\big\|{\Scal}_{2,\hh}\big(\sqrt{L}f\big)\big\|_{L^q(w)}.
\end{align}
Next, note that $h:=\sqrt{L}f$ is in the Hardy space $\mathbb{
H}^q_{\nabla L^{-\frac{1}{2}},p}(w)$ (see \cite{PA17}) for any $p\in \mathcal{W}_w(q_-(L),q_+(L))$, since, by Theorem \ref{thm:wRR}, $h\in L^p(w)$,  and $\|\nabla L^{-\frac{1}{2}}h\|_{L^q(w)}=\|\nabla f\|_{L^q(w)}<\infty$. Therefore, by \cite[Proposition 9.1]{PA17} we obtain that, for all 
$$
\max\left\{r_w,\frac{nr_w\widehat{p}_-(L)}{nr_w+\widehat{p}_-(L)}\right\}<q<\frac{p_+(L)}{s_w},$$ 
it holds that
\begin{align}\label{usedfinal2}
\big\|{\Scal}_{2,\hh}\big(\sqrt{L}f\big)\big\|_{L^q(w)}=
\|{\Scal}_{2,\hh}h\|_{L^q(w)}
\lesssim
\|\nabla L^{-\frac{1}{2}}h\|_{L^q(w)}=
\|\nabla f\|_{L^q(w)}.
\end{align}
This, together with \eqref{2}, allows us to conclude, for all $\max\left\{r_w,\frac{nr_w\widehat{q}_-(L)}{nr_w+\widehat{q}_-(L)}\right\}<q<\frac{q_+(L)}{s_w}$,  (recall that $\widehat{q}_-(L)=\widehat{p}_-(L)$ and $q_+(L)\leq p_+(L)$),
\begin{align*}
\|\wt \Grm_{\mathrm H}f\|_{L^q(w)}
\lesssim
\|\nabla f\|_{L^q(w)}.
\end{align*}
To finish the proof, note that in view of Lemma \ref{lem:wVSF} and Theorem \ref{thm:wRR},
we can improve  the upper bound of the interval where the above inequality holds up to $\widehat{q}_+(L)$ (assuming that $q_+(L)/s_w<\widehat{q}_+(L)$, otherwise there is nothing to prove). Indeed, we just need to observe that for $q$  such that $q_+(L)/s_w\leq q<\widehat{q}_+(L)$, we have that $q$ falls within the scope of those results.  
This follows since we assume that $\mathcal{W}_w(q_-(L),q_+(L))\neq \emptyset$, and 
by definition, it holds that
$\mathcal{W}_w(q_-(L),q_+(L))\subset \mathcal{W}_w(p_-(L),p_+(L))$. Then, we have that $\mathcal{W}_w(p_-(L),p_+(L))\neq \emptyset$ and 
\begin{multline*}
\max\left\{r_w,\frac{nr_w\widehat{p}_-(L)}{nr_w+\widehat{p}_-(L)}\right\}
\le
\max\left\{r_w,\widehat{p}_-(L)\right\}
=\max\left\{r_w,\widehat{q}_-(L)\right\}
\\
\le \max\left\{r_w,r_wq_-(L)\right\}
= r_wq_-(L)
<\frac{q_+(L)}{s_w}\leq q<\widehat{q}_+(L)\leq \widehat{p}_+(L).
\end{multline*}
This completes the prof.
\qed

\begin{remark}
The result given by Theorem \ref{thm:wtG} is, as far as we know, new even in the unweighted case, that is, when $w\equiv 1$. In that scenario it says that for every $q$ such that $\max\big\{1, \frac{n q_-(L)}{n+ q_-(L)}\big\}<q <q_+(L)$ and $f\in \mathcal{S}$, there holds
$$
\norm{\wt \Grm_{\mathrm H} f}_{L^q(\R^n)} \lesssim \norm{\nabla f}_{L^q(\R^n)}.
$$
The condition $\mathcal{W}_w(q_-(L),q_+(L))\neq \emptyset$ always holds when $w\equiv 1$, since 
 $q_-(L)<2<q_+(L)$, and, by definition, $(q_-(L),q_+(L))=\mathcal{W}_w(q_-(L),q_+(L))$.
\end{remark}



\section{Proof of Theorem \ref{thm:main}}\label{section:main-proof}
In this section we prove Theorem \ref{thm:main}. Given $f\in \mathcal{S}$ we set  $u(x,t):=e^{-t\sqrt L}f(x)$, for each $(x,t)\in\R_+^{n+1}$. It is well-known that $u(\cdot,t)\in L^2(\R^n)$ uniformly in $t>0$ since the Poisson semigroup is uniformly bounded on $L^2(\R^n)$ (this latter fact can be seen directly from the subordination formula \eqref{eqn:subord} below along with the uniform $L^2$-boundedness of the heat semigroup, see Section \ref{section:off}). This and Caccioppoli's inequality readily imply that $u\in W_{\textrm{loc}}^{1,2}(\R_+^{n+1})$ and also $\L u=0$ in the weak sense. Using standard holomorphic functional calculus techniques one can also see that $u(\cdot,t)\to f$ in $L^2(\R^n)$ as $t\to 0^+$. Thus  we are left with showing \eqref{Est-main:thm} and to this end, it suffices to individually bound the operators $\Ncal(\partial_t e^{-t\sqrt L}f)$ and $\Ncal(\nabla e^{-t\sqrt L}f)$. 

We first deal with  $\Ncal(\partial_t e^{-t\sqrt L}f)$. In the unweighted case, in \cite{May10}, the estimate of this operator relies on the characterization of Hardy spaces associated with $L$ via the non-tangential maximal function associated with the Poisson semigroup and the Riesz transform (characterization established in \cite{HMay09, HMayMc11}). Recently, the weighted Hardy spaces have been carefully studied in  \cite{BuiCaoKyYangYang, BuiCaoKyYangYang:II, MaPAII17, PA17}, including various characterization of weighted Hardy spaces via molecules, square functions, non-tangential maximal functions, Riesz transform etc. This weighted Hardy space theory enables us to treat the weighted estimate of $\Ncal(\partial_t e^{-t\sqrt L}f)$ by following the path laid down in \cite{May10}. More precisely we obtain the following result whose proof is given in Section \ref{section:proof-t}: 

\begin{proposition}\label{prop:part-t} 
	Let $w\in A_{\infty}$ be such that $\Wcal_w(p_-(L), p_+(L))\neq \emptyset$, and let $p$ be chosen so that $\max\bbr{r_w, \frac{nr_w\widehat p_-(L)}{nr_w+\widehat p_-(L)}}<p<\frac{p_+(L)}{s_w}$. Then, for any $f\in \mathcal{S}$,
	\begin{equation}\label{eq:timeD}
	\norm{\Ncal (\partial_t e^{-t\sqrt L}f)}_{L^p(w)} \lesssim \norm{\nabla f}_{L^p(w)}.
	\end{equation}
\end{proposition}

 Regarding the spatial derivatives, the following result establish the desired bound for $\Ncal(\nabla e^{-t\sqrt L}f)$:

\begin{proposition}\label{prop:NcalGrad}
	Let $w\in A_{\infty}$ be such that $\mathcal{W}_w(q_-(L),q_+(L))\neq \emptyset$. Then, for $\max\big\{r_w,\frac{nr_w\widehat q_-(L)}{nr_w+\widehat q_-(L)}\big\}<p<\frac{q_+(L)}{s_w}$ and $f\in \mathcal{S}$, there holds
	\begin{equation}\label{eq:weightedG}
	\norm{\Ncal (\nabla  \,e^{-t\sqrt L}f)}_{L^p(w)} \lesssim \norm{\nabla f}_{L^p(w)}.
	\end{equation}
\end{proposition}

The proof of this result is in Section \ref{section:proof-nabla}. Our method differs from the one in \cite{May10} and it is of independent interest. More precisely, when $w\equiv 1$, our proof provides an alternative approach (correcting some flaw) to \cite[Proof of Theorem 4.1, Steps II--VIII]{May10}, in which matters can be essentially reduced to estimate the inhomogeneous vertical square function $\wt \Grm_{\mathrm H}$  (see Theorem \ref{thm:wtG}) along with some similar ``homogeneous'' conical square function estimates proved in \cite{PA18} (see \eqref{usedfinal1} and \eqref{usedfinal2}).

\subsection{Proof of Proposition \ref{prop:part-t}}\label{section:proof-t}
We first introduce the non-tangential maximal function with respect to the Poisson semigroup defined by
\begin{equation}\label{eq:NcalPois}
\Ncal_{\pp}(g)(x):=\sup_{t>0} \br{\fint_{B(x,t)} |e^{-t\sqrt L}g(z)|^2 dz}^{\frac12},\quad x\in \R^n.
\end{equation}
The weighted norm inequalities for $\Ncal_{\pp}$ can be found in \cite[Proposition 7.1 (b)]{MaPAII17} and \cite[Theorem 3.7]{PA18}:
\begin{lemma}\label{lem:NpBoundedness}
Given $w\in A_{\infty}$ such that $\Wcal_w(p_-(L), p_+(L))\neq \emptyset$, then $\Ncal_{\pp}$ is bounded on $L^p(w)$ for all $p\in \Wcal_w(p_-(L), (p_+(L))^*)$.
\end{lemma}

We are now ready to prove Proposition \ref{prop:part-t}. First, it was shown in \cite[(4.25)]{May10} that 
\begin{equation}\label{eq:comp}
\Ncal(\partial_t e^{-t\sqrt L}f)(x) \lesssim \Ncal_{\pp}(\sqrt L f)(x),\quad \forall x\in \R^n.
\end{equation}
Then by Lemma \ref{lem:NpBoundedness}, we have, for  $p\in \Wcal_w(p_-(L),(p_+(L))^*)$,
\begin{equation}\label{eq:Lp}
 \norm{\Ncal_{\pp}(\sqrt L f)}_{L^p(w)}\lesssim \norm{\sqrt L f}_{L^p(w)}.
\end{equation}
Besides, \cite[Theorem 1.1]{PA17} shows that  for any $p,q\in \Wcal_w(p_-(L), p_+(L))$, the weighted Lebesgue space $L^p(w)$ and the Hardy space $H_{\Scal_{\hh},q}^p(w)$ (see the definition in \cite{PA17})  are isomorphic with equivalent norms. This and Lemma \ref{lem:NpBoundedness} readily give that for any fixed $q\in \Wcal_w(p_-(L), p_+(L))$,
\begin{equation}\label{eq:HpLp}
\Ncal_{\pp}: H_{\Scal_{\hh},q}^p(w) \to L^p(w), \quad \forall p\in \Wcal_w(p_-(L), p_+(L)).
\end{equation}
Furthermore,  by \cite[Theorems 3.9 and 3.11]{MaPAII17} we also have that
\begin{equation}\label{eq:H1L1}
\Ncal_{\pp}: H_{\Scal_{\hh},q}^1(w) \to L^1(w).
\end{equation}
Hence, in view of the interpolation result \cite[Theorem 5.1]{PA17}, by \eqref{eq:HpLp} and \eqref{eq:H1L1}, we get
\begin{equation*}
\Ncal_{\pp}: H_{\Scal_{\hh},q}^p(w) \to L^p(w), \quad 1\le  p <\frac{p_+(L)}{s_w},
\end{equation*}
and thus
\begin{equation}\label{eq:Lp-1}
 \norm{\Ncal_{\pp}(\sqrt L f)}_{L^p(w)}\lesssim \norm{\sqrt L f}_{H_{\Scal_{\hh},q}^p(w)}, \quad 1\le  p <\frac{p_+(L)}{s_w}.
\end{equation}
Next, from the weighted Hardy space Riesz transform characterization (see \cite[Propositions 9.1]{PA17}), it follows that, for all $f\in \mathcal{S}$,
\begin{equation}\label{eq:RT}
\norm{\sqrt L f}_{H_{\Scal_{\hh},q}^p(w)} \lesssim \norm{\nabla f}_{L^p(w)},  \quad \max\bbr{r_w, \frac{nr_w\widehat p_-(L)}{nr_w+\widehat p_-(L)}}<p<\frac{p_+(L)}{s_w}.
\end{equation}
Finally, combining \eqref{eq:comp}, \eqref{eq:Lp-1}, and \eqref{eq:RT}, we obtain that, for $\max\bbr{r_w, \frac{nr_w\widehat p_-(L)}{nr_w+\widehat p_-(L)}}<p<\frac{p_+(L)}{s_w}$, 
$$
\norm{\Ncal \big(\partial_t e^{-t\sqrt L}f\big)}_{L^p(w)} \lesssim  \norm{\Ncal_{\pp}(\sqrt L f)}_{L^p(w)}
\lesssim \norm{\sqrt L f}_{H_{\Scal_{\hh},q}^p(w)} \lesssim \norm{\nabla f}_{L^p(w)},
$$
for all $ f\in \mathcal{S}$. This completes the proof.\qed

\subsection{Proof of Proposition \ref{prop:NcalGrad}}\label{section:proof-nabla}

We split the proof into two steps. In \textbf{Step 1} we  obtain \eqref{eq:weightedG} for all  
$\max\{r_w,\widehat{q}_-(L)\}<p<\frac{q_+(L)}{s_w}$; and,  in \textbf{Step 2} we show that the same estimate holds in a bigger range, namely for all $\max\left\{r_w,\frac{nr_w\widehat q_-(L)}{nr_w+\widehat q_-(L)}\right\}<p<\frac{q_+(L)}{s_w}$. 

The following claims are 
common to both  steps, so we start by proving them.
\begin{claim}\label{claim:1}Under the hypothesis of Proposition  \ref{prop:NcalGrad}, for any $\max\{r_w,\widehat{q}_-(L)\}<q_0<q_+(L)/s_w$, there holds
$$
\left(\dashint_{B(x,t)} |\nabla e^{-t^2L} f|^2 dz\right)^{\frac{1}{2}}
\leq
\sum_{l\geq 1}e^{-c4^l}
\left(\dashint_{B(x,2^{l+1}t)} |\nabla  S_{t/\sqrt{2}}f|^{q_0} dw\right)^{\frac{1}{q_0}},
$$
where $ S_{t/\sqrt{2}}$ could be equal to $e^{-\frac{t^2}{4}L}$ or the identity.
\end{claim}

\begin{claim}\label{claim:2}Under the hypothesis of Proposition  \ref{prop:NcalGrad}, for any $\max\{r_w,\widehat{q}_-(L)\}<q_0<q_+(L)/s_w$, there holds
\begin{multline*}
\int_{0}^{\frac{1}{4}}u^{\frac{1}{2}}
\left(\dashint_{B(x,t)} |\nabla (e^{-\frac{t^2}{4u} L}-e^{-t^2L}) f(z)|^2 dz\right)^{\frac{1}{2}}\frac{du}{u}
\\
\lesssim
 \sum_{l\geq 1}e^{-c2^l} \br{\,\fint_{B(x,2^{l+1}t)} |\wt\Grm_{\hh}  f|^{q_0} dw}^{\frac1{q_0}}.
\end{multline*}
\end{claim}

\begin{claim}\label{claim:3}Under the hypothesis of Proposition  \ref{prop:NcalGrad}, there holds
$$
\int_{\frac{1}{4}}^{\infty}e^{-u}u^{\frac{1}{2}}
\left(\dashint_{B(x,t)} |\nabla (e^{-\frac{t^2}{4u} L}-e^{-t^2L}) f(z)|^2 dz\right)^{\frac{1}{2}}\frac{du}{u}
\lesssim
\int_{\frac{1}{4}}^{\infty}u\,e^{-u}\mathrm{G}^{2\sqrt{u}}_{\hh}(\sqrt{L}f)(x)\frac{du}{u},
$$
where
$$
\mathrm{G}^{2\sqrt{u}}_{\hh}(\sqrt{L}f)(x):=
\left(\int_{B(x,2\sqrt{u}s)}\int_{0}^{\infty}\big|s \nabla s \sqrt{L} e^{-s^2 L} \sqrt{L}f(z)\big|^2 \frac{dz\,ds}{s^{n+1}} \right)^{\frac{1}{2}}.
$$
\end{claim}

In order to prove these claims,  fix $p_0$ and $q$ such that $q_-(L)<p_0<2$, $r_wp_0<q_+(L)/s_w$, and  $\max\{q_0,r_wp_0\}<q<q_+(L)/s_w$.

\begin{proof}[Proof of Claim \ref{claim:1}]
First, apply the second estimate in  Proposition \ref{prop:Gsum} with $w\equiv 1$, $\lambda=1$, $s=t$, $q=2$,  $p=p_0$, and $H(z,t)=f(z)$ for all $(z,t)\in \R^{n+1}_+$. Then
\begin{align}\label{claim11}
\left(\dashint_{B(x,t)} |\nabla e^{-t^2L} f(z)|^2 dz\right)^{\frac{1}{2}}\lesssim \sum_{j\geq 1}e^{-c4^j}
\left(\dashint_{B(x,2^{j+1}t)} |\nabla e^{-\frac{t^2}{2}L}  f(z)|^{p_0} dz\right)^{\frac{1}{p_0}}.
\end{align}
Besides, note that since $w\in A_{\frac{q}{p_0}}$, by Lemma \ref{lemma:sin-con}, we have
\begin{align}\label{claim12}
\left(\dashint_{B(x,2^{j+1}t)} |\nabla e^{-\frac{t^2}{2}L}  f(z)|^{p_0} dz\right)^{\frac{1}{p_0}}
\lesssim \left(\dashint_{B(x,2^{j+1}t)} |\nabla e^{-\frac{t^2}{2}L}  f(z)|^{q} dw\right)^{\frac{1}{q}}.
\end{align}
Also, since $\max\{r_w,\widehat{q}_-(L)\}<q_0<q<q_+(L)/s_w\leq\widehat{q}_+(L)$,  again by Proposition \ref{prop:Gsum} (we use the first or the second estimate depending on whether $S_{t/\sqrt{2}}$ is the identity or $e^{-\frac{t^2}{4}L}$ respectively) with $\lambda=2^{j+1}\sqrt{2}$ and $s=t/\sqrt{2}$, we obtain
\begin{align*}
\left(\dashint_{B(x,2^{j+1}t)} |\nabla e^{-\frac{t^2}{2}L}  f(z)|^{q} dw\right)^{\frac{1}{q}}
&\lesssim
e^{-c4^j}\sum_{i\geq 1}e^{-c4^i}
\left(\dashint_{B(x,2^{j+i+2}t)} |\nabla S_{t/\sqrt{2}} f(z)|^{q_0} dw\right)^{\frac{1}{q_0}}
.
\end{align*}
This, \eqref{claim11} and \eqref{claim12}
imply 
\begin{align*}
\left(\dashint_{B(x,t)} |\nabla e^{-t^2L} f(z)|^2 dz\right)^{\frac{1}{2}}
&\lesssim
\sum_{j\geq 1}e^{-c4^j}\sum_{i\geq 1}e^{-c4^i}
\left(\dashint_{B(x,2^{j+i+2}t)} |\nabla S_{t/\sqrt{2}} f(z)|^{q_0} dw\right)^{\frac{1}{q_0}}
\\\nonumber
&
\lesssim
\sum_{l\geq 1}e^{-c2^l}
\left(\dashint_{B(x,2^{l+1}t)} |\nabla  S_{t/\sqrt{2}} f(z)|^{q_0} dw\right)^{\frac{1}{q_0}},
\end{align*}
which gives the desired estimate.
\end{proof}

\begin{proof}[Proof of Claim \ref{claim:2}]
We proceed as in Claim \ref{claim:1}. First apply the second estimate in Proposition \ref{prop:Gsum} with $w\equiv 1$, $\lambda=\sqrt{2}$, $s=t/\sqrt{2}$, $q=2$,  and $p=p_0$. Next, apply Lemma \ref{lemma:sin-con}; and finally, apply  the first estimate in Proposition \ref{prop:Gsum}, with $\lambda=2^{j+2}$, and $s=t/2$, to obtain for every $0<u<1/4$  
\begin{align}\label{1:II1}
&\left(\dashint_{B(x,t)} |\nabla (e^{-\frac{t^2}{4u} L}-e^{-t^2L}) f(z)|^2 dz\right)^{\frac{1}{2}}
\\\nonumber&\qquad
=
\left(\dashint_{B(x,t)} |\nabla e^{-\frac{t^2}{2}L} (e^{-\left(\frac{t^2}{4u}-\frac{t^2}{2}\right) L}-e^{-\frac{t^2}{2}L}) f(z)|^2 dz\right)^{\frac{1}{2}}
\\\nonumber&\qquad
\lesssim\sum_{j\geq 1}e^{-c4^j}
\left(\dashint_{B(x,2^{j+1}t)} |\nabla  e^{-\frac{t^2}{4}L} (e^{-\left(\frac{t^2}{4u}-\frac{t^2}{2}\right) L}-e^{-\frac{t^2}{2}L}) f(z)|^{p_0} dz\right)^{\frac{1}{p_0}}
\\\nonumber&\qquad
\lesssim\sum_{j\geq 1}e^{-c4^j}
\left(\dashint_{B(x,2^{j+1}t)} |\nabla e^{-\frac{t^2}{4}L} (e^{-\left(\frac{t^2}{4u}-\frac{t^2}{2}\right) L}-e^{-\frac{t^2}{2}L}) f(z)|^{q} dw\right)^{\frac{1}{q}}
\\\nonumber&\qquad
\lesssim\sum_{j\geq 1}e^{-c4^j}\sum_{i\geq 1}e^{-c4^i}
\left(\dashint_{B(x,2^{j+i+2}t)} |\nabla  (e^{-\left(\frac{t^2}{4u}-\frac{t^2}{2}\right) L}-e^{-\frac{t^2}{2}L}) f(z)|^{q_0} dw\right)^{\frac{1}{q_0}}.
\end{align}
Moreover, for $0<u<1/4$,
\begin{multline}\label{2:II1}
\big|\nabla  \big(e^{-(\frac{1}{4u}-\frac{1}{2}) t^2 L}- e^{-\frac{t^2}{2} L}\big)f(z)\big| 
\leq 
\int_{\frac{t}{\sqrt 2}}^{t\sqrt{\frac{1}{4u}-\frac{1}{2}}} \big|\partial_s \nabla  e^{-s^2 L} f(z)\big| ds
\\
\lesssim 
\int_{\frac{t}{\sqrt 2}}^{t\sqrt{\frac{1}{4u}}} \big|s^2\nabla  L e^{-s^2 L} f(z)\big| \frac{ds}{s}
\lesssim
\br{\int_{\frac{t}{2}}^{\infty} \big|s^2 \nabla  L e^{-s^2 L} f(z)\big|^2 \frac{ds}{s}}^{\frac12} \br{\log{{u}^{-1}}}^{\frac12}
\\
\lesssim
\,\br{\log{{u}^{-1}}}^{\frac12} \wt\Grm_{\hh,t} ( f)(z), 
\end{multline}
where $\wt\Grm_{\mathrm H,t}$ is defined as
\begin{align}\label{def:WTfG}
\wt\Grm_{\mathrm H,t}f(x):=\br{\int_{\frac{t}{2}}^{\infty} \big|s^2 \nabla  L e^{-s^2 L} f(z)\big|^2 \frac{ds}{s}}^{\frac12}.
\end{align}
 Therefore, by \eqref{1:II1} and \eqref{2:II1},
\begin{align}\label{lateruse}
&\left(\dashint_{B(x,t)} |\nabla (e^{-\frac{t^2}{4u} L}-e^{-t^2L}) f(z)|^2 dz\right)^{\frac{1}{2}}
\\\nonumber&\qquad\lesssim
 \sum_{j\geq 1} \sum_{i\geq 1} e^{-c2^{j+i+1}} \int_{0}^{\frac{1}{4}}  u^{\frac12}\br{\log{{u}^{-1}}}^{\frac12} \frac{du}{u}\br{\,\fint_{B(x,2^{j+i+2}t)} |\wt\Grm_{\hh,t}  f (z)|^{q_0} dw}^{\frac1{q_0}} 
 \\\nonumber&\qquad\lesssim
 \sum_{l\geq 1}e^{-c2^l} \br{\,\fint_{B(x,2^{l+1}t)} |\wt\Grm_{\hh,t}  f (z)|^{q_0} dw}^{\frac1{q_0}}.
\end{align}
This completes the proof of the present claim.
\end{proof}

\begin{proof}[Proof of Claim \ref{claim:3}]

First, note that for $1/4\leq u<\infty$,
\begin{align*}
\big|\nabla  \big(e^{-\frac{t^2}{4u}L}- e^{-t^2 L}\big)f(z)\big| &
\leq 
\int_{\frac{t}{2\sqrt{u}}}^{\frac{t}{\sqrt 2}}
 \big|\partial_s \nabla  e^{-s^2 L} f(z)\big| ds
\\&
\lesssim {u}^{\frac{1}{2}}
\br{\int_{\frac{t}{2\sqrt{u}}}^{\frac{t}{\sqrt 2}}\big|s \nabla s \sqrt{L} e^{-s^2 L} \sqrt{L}f(z)\big|^2 \frac{ds}{s}}^{\frac12}.
\end{align*}
Thus,
\begin{align*}
&\int_{\frac{1}{4}}^{\infty}e^{-u}u^{\frac{1}{2}}
\left(\dashint_{B(x,t)} |\nabla (e^{-\frac{t^2}{4u} L}-e^{-t^2L}) f(z)|^2 dz\right)^{\frac{1}{2}}\frac{du}{u}
\\&\qquad\lesssim
\int_{\frac{1}{4}}^{\infty}u\,e^{-u}
\left(\int_{B(x,2\sqrt{u}s)}\int_{0}^{\infty}\big|s \nabla s \sqrt{L} e^{-s^2 L} \sqrt{L}f(z)\big|^2 \frac{dz\,ds}{s^{n+1}} \right)^{\frac{1}{2}}\frac{du}{u}
\\\nonumber&\qquad=
\int_{\frac{1}{4}}^{\infty}u\,e^{-u}\mathrm{G}^{2\sqrt{u}}_{\hh}(\sqrt{L}f)(x)\frac{du}{u},
\end{align*}
which is the desired estimate.
\end{proof}
	
\medskip

Once we have proved the claims, we can start discussing the two cases into which we split the proof of Proposition \ref{prop:NcalGrad}.

\medskip

\noindent\textbf{Step 1:}
$\max\{r_w,\widehat{q}_-(L)\}<p<\frac{q_+(L)}{s_w}$.

First of all, note that proceeding similarly  as in \cite[(4.25)]{May10} one can show that for every fixed $x\in \R^n$
\begin{align}\label{original}
\Ncal(\nabla  e^{-t\sqrt L}f)(x)  
&\lesssim
\sup_{t>0} \br{\,\fint_{B(x,t)} |\nabla e^{-t\sqrt L} f(z)|^2 dz}^{\frac12}
\\ \nonumber
&\le
\sup_{t>0} 
\left(\dashint_{B(x,t)} |\nabla e^{-t^2L} f(z)|^2 dz\right)^{\frac{1}{2}}
\\ \nonumber
&\qquad \quad
+
\sup_{t>0} \left(\dashint_{B(x,t)} |\nabla (e^{-t\sqrt L}-e^{-t^2L}) f(z)|^2 dz\right)^{\frac{1}{2}}
&
\\ \nonumber
&=:I + II.
\end{align}

Since $\max\{r_w,\widehat{q}_-(L)\}<p<\frac{q_+(L)}{s_w}$, we can choose $q_0$ so that $\max\{r_w,\widehat{q}_-(L)\}<q_0<p<\frac{q_+(L)}{s_w}$.
Hence by Claim \ref{claim:1} with $S_t$ equal to the identity, we obtain 
\begin{align}\label{termI}
I&
\lesssim
\sup_{t>0}
\sum_{l\geq 1}e^{-c2^l}
\left(\dashint_{B(x,2^{l+1}t)} |\nabla  f(z)|^{q_0} dw\right)^{\frac{1}{q_0}}
\lesssim
\mathcal{M}_{q_0}^w(\nabla  f)(x).
\end{align}
On the other hand, note that applying the following subordination formula:
\begin{align}\label{eqn:subord}
e^{-t\sqrt{L}}f=C\int_{0}^{\infty}u^{\frac{1}{2}}e^{-u}e^{-\frac{t^2}{4u}L}f\frac{du}{u},
\end{align}
and Minkowski's integral inequality, we have 
\begin{multline}\label{termII}
II\lesssim 
\sup_{t>0}\int_{0}^{\frac{1}{4}}u^{\frac{1}{2}}
\left(\dashint_{B(x,t)} |\nabla (e^{-\frac{t^2}{4u} L}-e^{-t^2L}) f(z)|^2 dz\right)^{\frac{1}{2}}\frac{du}{u}
\\
+
\sup_{t>0} \int_{\frac{1}{4}}^{\infty}e^{-u}u^{\frac{1}{2}}
\left(\dashint_{B(x,t)} |\nabla (e^{-\frac{t^2}{4u} L}-e^{-t^2L}) f(z)|^2 dz\right)^{\frac{1}{2}}\frac{du}{u}=:II_1+II_2.
\end{multline}
In order to estimate $II_1$, we apply Claim \ref{claim:2}:
\begin{multline}\label{lateruse-2}
II_1
\lesssim
\sup_{t>0} \sum_{l\geq 1}e^{-c2^l} \br{\,\fint_{B(x,2^{l+1}t)} |\wt\Grm_{\hh,t}  f (z)|^{q_0} dw}^{\frac1{q_0}} 
 \\
 \lesssim
 \sup_{t>0} \sum_{l\geq 1}e^{-c2^l}  \br{\,\fint_{B(x,2^{l+1}t)} |\wt\Grm_{\hh}  f (z)|^{q_0} dw}^{\frac1{q_0}} \lesssim
 \Mcal_{q_0}^w (\wt\Grm_{\hh} f )(x),
\end{multline}
where $\wt\Grm_{\hh}$ is
the inhomogeneous vertical square function defined in \eqref{eq:wtG}. On the other hand, for $1/4\leq u<\infty$, one can see that
\begin{multline*}
\big|\nabla  \big(e^{-\frac{t^2}{4u}L}- e^{-t^2 L}\big)f(z)\big| 
\leq 
\int_{\frac{t}{2\sqrt{u}}}^{\frac{t}{\sqrt 2}}
 \big|\partial_s \nabla  e^{-s^2 L} f(z)\big| ds
\\
\lesssim
\br{\int_{\frac{t}{2\sqrt{u}}}^{\frac{t}{\sqrt 2}}\big|s \nabla s \sqrt{L} e^{-s^2 L} \sqrt{L}f(z)\big|^2 \frac{ds}{s}}^{\frac12} \br{\log{{u}^{\frac{1}{2}}}}^{\frac12}.
\end{multline*}
Next,  applying Claim \ref{claim:3} we get
\begin{align*}
II_2\lesssim
\int_{\frac{1}{4}}^{\infty}u\,e^{-u}\mathrm{G}^{2\sqrt{u}}_{\hh}(\sqrt{L}f)(x)\frac{du}{u}.
\end{align*}
This, \eqref{original}, \eqref{termI}, \eqref{termII}, and \eqref{lateruse-2} imply that, for $\max\left\{r_w,\widehat{q}_-(L)\right\}<q_0<p$,
\begin{align}\label{inequalityw}
\Ncal(\nabla  e^{-t\sqrt L}f)(x)  
\leq 
\mathcal{M}_{q_0}^w(\nabla  f)(x)+ \Mcal_{q_0}^w (\wt\Grm_{\hh} f )(x)
+
\int_{\frac{1}{4}}^{\infty}u\,e^{-u}\mathrm{G}^{2\sqrt{u}}_{\hh}(\sqrt{L}f)(x)\frac{du}{u}.
\end{align}
Consequently, for all $\max\left\{r_w,\widehat{q}_-(L)\right\}<p<\frac{q_+}{s_w}$,  by the boundedness of $\mathcal{M}_{q_0}^w$ on $L^p(w)$, recall that $p/q_0>1$, and change of angles (see \cite[Proposition 3.2]{MaPAI17}), we conclude that
\begin{align*}
&\|\Ncal(\nabla  e^{-t\sqrt L}f)\|_{L^p(w)}
\\
&\qquad\lesssim 
\|\mathcal{M}_{q_0}^w(\nabla  f)\|_{L^p(w)}+ \|\Mcal_{q_0}^w (\wt\Grm_{\hh} f )\|_{L^p(w)}
+
\int_{\frac{1}{4}}^{\infty}u\,e^{-u}\|\mathrm{G}^{2\sqrt{u}}_{\hh}(\sqrt{L}f)\|_{L^p(w)}\frac{du}{u}
\\
&\qquad\lesssim 
\|\nabla  f\|_{L^p(w)}+ \|\wt\Grm_{\hh} f \|_{L^p(w)}
+
\|\mathrm{G}_{\hh}(\sqrt{L}f)\|_{L^p(w)}
\\
&\qquad\lesssim 
\|\nabla  f\|_{L^p(w)},
\end{align*}
where  the last inequality follows from Theorem \ref{thm:wtG}, \eqref{usedfinal1}, and \eqref{usedfinal2}. This completes the proof  of \textbf{Step 1}.

\medskip

\noindent\textbf{Step 2:} $\max\left\{r_w,\frac{nr_w\widehat q_-(L)}{nr_w+\widehat q_-(L)}\right\}<p<\frac{q_+(L)}{s_w}$ .

We start by observing that \textbf{Step 1} leads to 
\begin{align}\label{interpol1}
\Ncal(\nabla  e^{-t\sqrt L}):\dot{W}^{1,q}(w)\rightarrow L^q(w),\quad \max\{r_w,\widehat{q}_-(L)\}<q<\frac{q_+(L)}{s_w},
\end{align}
where  $\dot{W}^{1,q}(w)=\overline{\big\{f\in \mathcal{S}:\nabla f\in L^q(w)\big\}}$, with the closure taken with respect to the seminorm $f\to \|\nabla f\|_{L^q(w)}$.
With this in hand,  by interpolation (see \cite{Badr}), we can conclude \eqref{eq:weightedG} for all $
\max\left\{r_w,\frac{r_wn\widehat{q}_-(L)}{r_wn+\widehat{q}_-(L)}\right\}<p<\frac{q_+(L)}{s_w}
$ provided that for every $\widetilde{p}$ and  $r_0>r_w$ such that
$\widehat{q}_-(L)<\widetilde{p}<q_+(L)/s_w$, $r_wq_-(L)<r_0q_-(L)<q_+(L)/s_w$, and  $\widetilde{q}:=\max\left\{r_0,\frac{nr_0\widetilde{p}}{nr_0+\widetilde{p}}\right\}$, we show that
\begin{align}\label{interpol2}
\Ncal(\nabla  e^{-t\sqrt L}):\dot{W}^{1,\widetilde{q}}(w)\rightarrow L^{\widetilde{q},\infty}(w).
\end{align}

\medskip

In order to prove \eqref{interpol2} fix $f\in \mathcal{S}$, and fix $\widetilde{p}$ and $r_0$ satisfying the above restrictions. Furthermore, fix  
$q_2$ and $p_1$ so that $q_-(L)<q_2<2$, $r_0q_2<\frac{q_+(L)}{s_w}$, and $\max\{\widetilde{p},r_0q_2\}<p_1<q_+(L)/s_w$. Note that in particular
\begin{align}\label{p_1choice}
w\in A_{\frac{p_1}{q_2}}\cap RH_{\left(\frac{q_+(L)}{p_1}\right)'}, \quad \textrm{and}\quad p_1>\widetilde{q}.
\end{align}
Moreover, note by Claim \ref{claim:1} with $S_t=e^{-\frac{t^2}{4}L}$, and Claims \ref{claim:2} and \ref{claim:3}, we obtain that, for any $q_0$ such that $\max\{r_w,\widehat{q}_-(L)\}<q_0<\max\{r_0,\widetilde{p}\}<q_+(L)/s_w$ and any function $h\in L^2(\R^n)$,
\begin{multline}\label{bettersplit}
\Ncal(\nabla  e^{-t\sqrt L}h)(x)  
\lesssim \sum_{l\geq 1}e^{-c2^l}
\sup_{t>0}
\left(\dashint_{B(x,2^{l+2}t)} |\nabla e^{-t^2L} h(z)|^{q_0} dw\right)^{\frac{1}{q_0}}
\\\qquad+
 \sum_{l\geq 1}e^{-c2^l} \sup_{t>0}\br{\,\fint_{B(x,2^{l+2}t)} |\wt\Grm_{\hh,t}  h(z)|^{q_0} dw}^{\frac1{q_0}} 
+
\int_{\frac{1}{4}}^{\infty}u\,e^{-u}\mathrm{G}^{2\sqrt{u}}_{\hh}(\sqrt{L}h)(x)\frac{du}{u}
\\=:
\sum_{l\geq 1}e^{-c2^l}\mathfrak{O}_{1,l}h(x)+
\sum_{l\geq 1}e^{-c2^l}\mathfrak{O}_{2,l}h(x)+
\mathfrak{O}_{3}h(x).
\end{multline}

Next we observe that to obtain \eqref{interpol2} it suffices to prove for every $\alpha>0$ that
\begin{align}\label{proofconcluded}
w\left(\left\{x\in \R^n: \Ncal(\nabla  e^{-t\sqrt L}f)(x)>\alpha\right\}\right)\lesssim \frac{1}{\alpha^{\widetilde{q}}}\int_{\R^n}|\nabla f(x)|^{\widetilde{q}} dw.
\end{align}
In order to obtain this inequality, fix $\alpha>0$ and consider the  Calderón-Zygmund decomposition for the function $f$ at height $\alpha$ given by Lemma \ref{lemma:C-Z-decomposition} with $p_0=\widetilde{q}$. 
Let $\{B_i\}_{i}$ be the corresponding collection of balls, and define, for $M\in \N$ arbitrarily large, 
\[
\mathcal{B}_{r_{B_i}}:=(I-e^{-r_{B_i}^2L})^M,\qquad \mathcal{A}_{r_{B_i}}:=I-\mathcal{B}_{r_{B_i}}=\sum_{k=1}^MC_{k,M}e^{-kr_{B_i}^2L}.
\]
 Hence,
\begin{align}\label{decompositionf}
f=g+\sum_{i}\mathcal{A}_{r_{B_i}}b_i+\sum_{i}\mathcal{B}_{r_{B_i}}b_i=:g+\widetilde{b}+\sum_i\widehat{b}_i=:g+\widetilde{b}+\widehat{b}.
\end{align}
and then 
\begin{align}\label{maintimederivative}
&w\big(\big\{x\in \R^n:  \, \Ncal\big(\nabla  e^{-t\sqrt L}f\big)(x)>\alpha\big\}\big)
\leq
w\big(\big\{x\in \R^n:\Ncal\big(\nabla  e^{-t\sqrt L} g\big)(x)>\alpha/5\big\}\big)
\\\nonumber&
\qquad\qquad
+
w\big(\big\{x\in \R^n:\Ncal\big(\nabla  e^{-t\sqrt L} \,\widetilde{b}\big)(x)\alpha/5\big\}\big)
\\\nonumber&
\qquad\qquad
+\sum_{m=1}^{2}
w\big(\big\{x\in \R^n:C\sum_{l\geq 1}e^{-c2^l}\mathfrak{O}_{m,l}\widehat{b}(x)>\alpha/5\big\}\big)
\\\nonumber&
\qquad\qquad
+
w\big(\big\{x\in \R^n:C\mathfrak{O}_{3}\widehat{b}(x)>\alpha/5\big\}\big)
\\\nonumber
&
\qquad   =:\mathcal I+ \mathcal{II}+\sum_{m=1}^2\mathcal{III}_{m}+\mathcal{IV}.
\end{align}

Before starting with the estimate of the above terms, we make a couple of observations for later use. First, take 
 $1< \mathfrak{q}<\infty$ and $h\in L^{\mathfrak{q}'}(w)$ such that $\|h\|_{L^{\mathfrak{q}'}(w)}= 1$ and recall the definition of  $\mathcal{M}^w$  in \eqref{weightedmaximal}. Then,
using a Kolmogorov type inequality (see \cite[Exercise 2.1.5]{GrafakosI}), the fact that $\mathcal{M}^w$ is bounded from $L^1(w)$ to $L^{1,\infty}(w)$ since $w\in A_\infty$ and hence it gives rise to a doubling measure,  and \eqref{C-Z:sum},
we have that
\begin{multline}\label{maximal-u}
\Big(\sum_{i\in \N} \int_{B_i}\mathcal{M}^w(|h|^{\mathfrak{q}'})(x)^{\frac{1}{\mathfrak{q}'}}w(x)dx\Big)^{\mathfrak{q}}
\lesssim
\Big(\int_{\cup_{i\in \N}B_i}\mathcal{M}^w(|h|^{\mathfrak{q}'})(x)^{\frac{1}{\mathfrak{q}'}}w(x)dx\Big)^{\mathfrak{q}}
\\
\lesssim
w(\cup_{i\in \N}B_i)\|h\|_{L^{\mathfrak{q}'}(w)}^\mathfrak{q}\lesssim \frac{1}{\alpha^{\widetilde{q}}}\int_{\R^n}|\nabla f(x)|^{\widetilde{q}}w(x)dx.
\end{multline}
Second, note that for $p_2:=\max\{r_0,\widetilde{p}\}$,
 we have that $q_0<p_2$. Assuming momentarily that  $p_2<(\widetilde{q})_w^*$, by \eqref{C-Z:extrab}, we get that
\begin{align}\label{CZ-ptildeexta}
 \left(\dashint_{B_i}|b_i(x)|^{p_2}dw\right)^{\frac{1}{p_2}}\lesssim \alpha r_{B_i}.
\end{align}
In order to see that $p_2<(\widetilde{q})_w^*$, we first consider the case $p_2=r_0$. Then, $\widetilde{q}=\max\left\{r_0,\frac{nr_0\widetilde{p}}{nr_0+\widetilde{p}}\right\}=r_0$, and thus 
$p_2=r_0<(r_0)_w^*=(\widetilde{q})_w^*$. On the other had, if $p_2=\widetilde{p}$, 
we may assume that $nr_w>\widetilde{q}$ ---otherwise we trivially have $\widetilde{p}=p_2<(\widetilde{q})_w^*=\infty$. Besides,
\begin{multline*}
\frac{1}{(\widetilde{q})_w^*}=\frac{1}{\widetilde{q}}-\frac{1}{nr_{w}}
<\frac{nr_0+\widetilde{p}}{nr_0\widetilde{p}}-\frac{1}{nr_{w}}
=\frac{1}{\widetilde{p}}+\frac{1}{nr_0}-
\frac{1}{nr_{w}}
\\
=
\frac{1}{\widetilde{p}}+\frac{1}{n}\left(\frac{1}{r_0}-\frac{1}{r_w}\right)
\leq 
\frac{1}{\widetilde{p}}=\frac{1}{p_2}.
\end{multline*}

\medskip

Now, we are ready to estimate the terms in \eqref{maintimederivative}.
In order to estimate $\mathcal{I}$, first recall that $\widetilde{q}<p_1$ and $\max\{r_w,\widehat{q}_-(L)\}<p_1<q_+(L)/s_w$ (see \eqref{p_1choice}). Then, 
by Chebyshev's inequality, \eqref{interpol1} and Lemma \ref{lemma:C-Z-decomposition}, we get
\begin{align}\label{firstermsum}
\mathcal{I}
&\lesssim
 \frac{1}{\alpha^{p_1}}\int_{\R^n}\Ncal\big(\nabla  e^{-t\sqrt L} g\big)^{p_1}dw
 \lesssim 
 \frac{1}{\alpha^{p_1}}\int_{\R^n}|\nabla g|^{p_1}dw
\lesssim
   \frac{1}{\alpha^{\widetilde{q}}}\int_{\R^n}|\nabla f|^{\widetilde{q}}dw.
\end{align}
Next we estimate
$\mathcal{II}$. Consider $p_2$ defined as in \eqref{CZ-ptildeexta} and apply Chebyshev's inequality and \eqref{interpol1} (recall that $\max\{r_w,\widehat{q}_-(L)\}<p_2<p_1<q_+(L)/s_w$ ). Thus
\begin{align}\label{prevtildeb}
&\mathcal{II}\lesssim\frac{1}{\alpha^{p_2}}\int_{\R^n} \Ncal\big(\nabla  e^{-t\sqrt L}\,\widetilde{b} \big)^{p_2}dw
\lesssim\frac{1}{\alpha^{p_2}}\int_{\R^n} |\nabla \widetilde{b}|^{p_2}dw.
\\\nonumber
\end{align}
We claim that 
\begin{align}\label{tildeb}
\int_{\R^n} |\nabla \widetilde{b}|^{p_2}dw\lesssim \frac{\alpha^{p_2}}{\alpha^{\widetilde{q}}}\int_{\R^n}|\nabla f|^{\widetilde{q}}dw.
\end{align}
Indeed, 
by the definition of $\mathcal{A}_{r_{B_i}}$, 
\begin{multline*}
\int_{\R^n} |\nabla \widetilde{b}|^{p_2}dw
\lesssim
\int_{\R^n} \bigg|\nabla\bigg(\sum_{i\in \N}\sum_{k=1}^{M}{C_{k,M}}e^{-kr_{B_i}^2L}b_i\bigg)\bigg|^{p_2}dw
\\
\lesssim
\sup_{\|h\|_{L^{p_2'}(w)}=1}
\bigg(\sum_{k=1}^{M}\sum_{i\in \N}\int_{\R^n} \bigg|\sqrt{k}r_{B_i}\nabla  e^{-kr_{B_i}^2L}\bigg(\frac{b_i}{r_{B_i}}\bigg)\bigg|\,|h|dw\bigg)^{p_2}.
\end{multline*}
Besides,  using that $\{\sqrt{t}\nabla  e^{-tL}\}_{t>0}$ satisfies $L^{p_2}(w)$-$L^{p_2}(w)$ off-diagonal estimates on balls; and  by \eqref{doublingcondition} and \eqref{CZ-ptildeexta}, we have
\begin{align*}
&\int_{\R^n} \bigg|\sqrt{k}r_{B_i}\nabla e^{-kr_{B_i}^2L}\bigg(\frac{b_i}{r_{B_i}}\bigg)\bigg|\,|h|dw
\\
&\lesssim \sum_{j\geq 1}w(2^{j+1}B_i)\bigg(\dashint_{C_j(B_i)}\bigg|\sqrt{k}r_{B_i}\nabla e^{-kr_{B_i}^2L}\bigg(\frac{b_i}{r_{B_i}}\bigg)\bigg|^{p_2}dw\bigg)^{\frac{1}{p_2}}
\bigg(\dashint_{C_j(B_i)}|h|^{p_2'}dw\bigg)^{\frac{1}{p_2'}}
\\\nonumber
&\lesssim
\sum_{j\geq 1}e^{-c4^j}w(B_i)\bigg(\dashint_{B_i} \bigg|\frac{b_i(x)}{r_{B_i}}\bigg|^{p_2}dw\bigg)^{\frac{1}{p_2}}\,\inf_{x\in B_i}\mathcal{M}^w(|h|^{p_2'})(x)^{\frac{1}{p_2'}}
\\&\lesssim
\alpha\int_{B_i}\mathcal{M}^w(|h|^{p_2'})^{\frac{1}{p_2'}}dw.
\end{align*}
Hence, by \eqref{maximal-u} with $\mathfrak{q}=p_2$, we have
\begin{align*}
\int_{\R^n} |\nabla \widetilde{b}|^{p_2}dw
&\lesssim\alpha^{p_2}\sup_{\|h\|_{L^{p_2'}(w)}=1}\bigg(\int_{B_i}\mathcal{M}^w(|h|^{p_2'})^{\frac{1}{p_2'}}dw\bigg)^{p_2}
\lesssim \frac{\alpha^{p_2}}{\alpha^{\widetilde{q}}}\int_{\R^n}|\nabla f(x)|^{\widetilde{q}}dw.
\end{align*}
Consequently, \eqref{tildeb} holds. In view of \eqref{prevtildeb}, that implies
\begin{align}\label{termcalII}
\mathcal{II}\lesssim  \frac{1}{\alpha^{\widetilde{q}}}\int_{\R^n}|\nabla f(x)|^{\widetilde{q}}dw.
\end{align}

We continue by estimating $\mathcal{IV}$.
To this end,  write  $b=\sum_i b_i$  so that $\widehat{b}=b-\widetilde{b}$, and note that 
\begin{multline*}
\mathcal{IV}
\leq  
w\big(\big\{x\in \R^n:C\mathfrak{O}_{3}b(x)>\alpha/10\big\}\big)
\\
+
 w\big(\big\{x\in \R^n:C\mathfrak{O}_{3}\widetilde{b}(x)>\alpha/10\big)\big\}
=: \mathcal{IV}_1+\mathcal{IV}_2.
\end{multline*}
In order to estimate $\mathcal{IV}_1$ apply Chebyshev's inequality, Minkowski's integral inequality, change of angles (\cite[Proposition 3.2]{MaPAI17}),  \eqref{usedfinal1}, and \eqref{usedfinal2}, to get
 \begin{multline*}
\mathcal{IV}_1
\lesssim
\frac{1}{\alpha^{\widetilde{q}}}\bigg(\int_{\frac{1}{4}}^{\infty}
e^{-cu}\big\|\Grm_{\hh}^{2\sqrt{u}}\big(\sqrt{L}\,b\big)\big\|_{L^{\widetilde{q}}(w)}
\frac{du}{u}\bigg)^{\widetilde{q}}
\\
\lesssim
\frac{1}{\alpha^{\widetilde{q}}}\big\|\Grm_{\hh}\big(\sqrt{L}\, {b}\big)\big\|_{L^{\widetilde{q}}(w)}^{\widetilde{q}}\lesssim
\frac{1}{\alpha^{\widetilde{q}}} \left\|\nabla b\right\|_{L^{\widetilde{q}}(w)}^{\widetilde{q}}.
\end{multline*}
Besides, note that by \eqref{C-Z:b} and \eqref{C-Z:sum}
\begin{align*}
\left\|\nabla b\right\|_{L^{\widetilde{q}}(w)}^{\widetilde{q}}\lesssim
\sum_{i}\int_{B_i}|\nabla b_i|^{\widetilde{q}} dw
\lesssim \alpha^{\widetilde{q}}
\sum_{i}w(B_i)
\lesssim \int_{\R^n}|\nabla f|^{\widetilde{q}} dw.
\end{align*}
Hence,
\begin{align*}
\mathcal{IV}_1\lesssim
\frac{1}{\alpha^{\widetilde{q}}}\int_{\R^n}|\nabla f|^{\widetilde{q}} dw.
\end{align*}
As for the estimate of $\mathcal{IV}_2$,  apply again Chebyshev's inequality, Minkowski's integral inequality, change of angles, \eqref{usedfinal1}, and \eqref{usedfinal2}, to get 
\begin{multline*}
\mathcal{IV}_2
\lesssim
\frac{1}{\alpha^{p_2}}\bigg(\int_{\frac{1}{4}}^{\infty}
e^{-cu}\big\|\Grm_{\hh}^{2\sqrt{u}}\big(\sqrt{L}\, \widetilde{b}\,\big)\big\|_{L^{p_2}(w)}
\frac{du}{u}\bigg)^{p_2}
\\
\lesssim
\frac{1}{\alpha^{p_1}}
\big\|\Grm_{\hh}\big(\sqrt{L}\, \widetilde{b}\,\big)\big\|_{L^{p_2}(w)}^{p_2}
\lesssim
\frac{1}{\alpha^{p_2}}
\|\nabla \widetilde{b}\,\|_{L^{p_2}(w)}^{p_2}
\end{multline*}
Thus, by \eqref{tildeb},
$$
\mathcal{IV}_2\lesssim 
\frac{1}{\alpha^{\widetilde{q}}}\int_{\R^n}|\nabla f|^{\widetilde{q}} dw.
$$
Collecting the estimates for $\mathcal{IV}_1$ and $\mathcal{IV}_2$ we conclude that
\begin{align}\label{termIIItimederivative}
\mathcal{IV}
\lesssim
\frac{1}{\alpha^{\widetilde{q}}}\int_{\R^n}|\nabla f|^{\widetilde{q}}dw.
\end{align}

It remains to estimate
$
\mathcal{III}_m$, for $m=1,2$. Note that by  \eqref{doublingcondition} and \eqref{C-Z:sum},
\begin{align}\label{termII_2}
\mathcal{III}_{m}
\leq &
w\Big(\bigcup_{i} 16B_i\Big)+
w\bigg(\bigg\{x\in \R^n\setminus \cup_{i} 16B_i:C\sum_{l\geq 1}e^{-c2^l}\mathfrak{O}_{m,l}\widehat{b}(x)>\alpha/5\bigg\}\bigg)
\\\nonumber
 \lesssim &
\frac{1}{\alpha^{\widetilde{q}}}\int_{\R^n}|\nabla f|^{\widetilde{q}}dw+\sum_{l\geq 1}
w\bigg(\bigg\{x\in \R^n\setminus \cup_{i} 16B_i:\mathfrak{O}_{m,l}\widehat{b}(x)>\frac{e^{c2^l}\alpha}{C2^{l}}
\bigg\}\bigg)
\\\nonumber
=: &
\frac{1}{\alpha^{\widetilde{q}}}\int_{\R^n}|\nabla f|^{\widetilde{q}}dw+\sum_{l\geq 1}\mathcal{III}_{m,l}.
\end{align}
Applying Chebyshev's inequality,  duality, and Hölder's inequality, it follows that
\begin{align}\label{sum:termII_2}
&\mathcal{III}_{m,l}\lesssim
\frac{e^{-c2^l}}{\alpha^{p_2}}
\int_{\R^n\setminus \cup_{i} 16B_i}
|\mathfrak{O}_{m,l}\widehat{b}|^{p_2}dw
\\\nonumber&\,\,
\lesssim
\frac{e^{-c2^l}}{\alpha^{p_2}}
\bigg(\sup_{\|h\|_{L^{p_2'}(w)}=1}\sum_{i}\sum_{j\geq 4}
\bigg(\int_{C_j(B_i)}
|\mathfrak{O}_{m,l}\widehat{b}_i|^{p_2}dw\bigg)^{\frac{1}{p_2}}\|h\chi_{C_j(B_i)}\|_{L^{p_2'}(w)}\bigg)^{p_2}
\\\nonumber &\,\,
=:
\frac{e^{-c2^l}}{\alpha^{p_2}}\bigg(\sup_{\|h\|_{L^{p_2'}(w)}=1}\sum_{i}\sum_{j\geq 4}I_{m,l}^{ij}\,\|h\chi_{C_j(B_i)}\|_{L^{p_2'}(w)}\bigg)^{p_2}.
\end{align}

\begin{claim}\label{claim}
There exist $\Theta, M_0>1$ such that for every $2\,M>\max\{\Theta, M_0\}$ 
\begin{align*}
I_{m,l}^{ij}\leq C 2^{lC_M}\alpha w(2^{j+1}B_i)^{\frac{1}{p_2}}2^{-j(2M-\Theta)}, \quad m=1,2.
\end{align*}
\end{claim}

Assuming this momentarily, in view of  \eqref{sum:termII_2}, for $2M>\max\left\{\Theta +nr_{w},M_0 \right\}$,  and by \eqref{maximal-u} with $\mathfrak{q}=p_2$ we get, for $m=1,2$,
\begin{align*}
\mathcal{III}_{m,l}
\lesssim e^{-c2^l} 
\bigg(\sum_{i} \int_{B_i}\mathcal{M}^w(|h|^{p_2'})^{\frac{1}{p_2'}}dw\bigg)^{p_2}
\lesssim e^{-c2^l}
\frac{1}{\alpha^{\widetilde{q}}}\int_{\R^n}|\nabla f|^{\widetilde{q}}dw.
\end{align*}
Therefore, by  \eqref{termII_2}, for $m=1,2$,
\begin{align*}
\mathcal{III}_{m}
\lesssim \sum_{l\geq 1} e^{-c2^l} 
\frac{1}{\alpha^{\widetilde{q}}}\int_{\R^n}|\nabla f|^{\widetilde{q}}dw
\lesssim
\frac{1}{\alpha^{\widetilde{q}}}\int_{\R^n}|\nabla f|^{\widetilde{q}}dw.
\end{align*}
By the above inequality, \eqref{maintimederivative}, \eqref{firstermsum},\eqref{termcalII}, and \eqref{termIIItimederivative}, we get \eqref{proofconcluded}. This leads to \eqref{interpol2} what in turn, as we have already observed, finishes the proof modulo Claim \ref{claim}. \qed

\begin{proof}[Proof of Claim \ref{claim}, $m=1$]	
Note that
\begin{align*}
{I}_{1,l}^{ij}&
\lesssim
\bigg(\int_{C_j(B_i)}\bigg(\sup_{0<t<2^{j-l-3}r_{B_i}}\bigg(\dashint_{B(x,2^{l+2}t)}
\big|\nabla e^{-t^2L}\widehat{b}_i\big|^{q_0}dw\bigg)^{\frac{1}{q_0}}\bigg)^{p_2}dw\bigg)^{\frac{1}{p_2}}
\\
&\,\,+
\bigg(\int_{C_j(B_i)}\bigg(\sup_{t\geq 2^{j-l-3}r_{B_i}}\bigg(\dashint_{B(x,2^{l+2}t)}
\big|\nabla e^{-t^2L}\widehat{b}_i\,\big|^{q_0}dw\bigg)^{\frac{1}{q_0}}\bigg)^{p_2}dw\bigg)^{\frac{1}{p_2}}
\\&=:\mathfrak{C}_1+\mathfrak{C}_2.
\end{align*}

In order to estimate $\mathfrak{C}_1$, we use functional calculus. Take $\phi(z,t):=e^{-t^2 z}(1-e^{-r_{B_i}^2 z})^M$, then $\phi(z,t)$ is holomorphic in the open sector $\Sigma_\mu=\{z\in\mathbb{C}\setminus\{0\}:|{\rm arg} (z)|<\mu\}$ and satisfies $|\phi(z,t)|\lesssim |z|^M\,(1+|z|)^{-2M}$ (with implicit constant depending on $\mu$, $t>0$, $r_{B_i}$, and $M$) for every $z\in\Sigma_\mu$. We can check that  for every $z\in \Gamma=\partial\sum_{\frac{\pi}{2}-\theta}$,
$$
|\eta(z,t)| \lesssim \frac{r_{B_i}^{2M}}{(|z|+t^2)^{M+1}}.
$$
Now fix $x\in C_{j}(B_i)$, $j\geq 4$, and $0<t<2^{j-l-3}r_{B_i}$, 
then $B(x,2^{l+2}t)\subset 2^{j+2}B_i\setminus 2^{j-1}B_i$. This and Minkowski's integral inequality imply
\begin{align*}
&\bigg(\dashint_{B(x,2^{l+2}t)}
\big|\nabla e^{-t^2L}\widehat{b}_i\big|^{q_0}dw\bigg)^{\frac{1}{q_0}}
=
\bigg(\dashint_{B(x,2^{l+2}t)}
\big|\nabla \phi(L,t)b_i\big|^{q_0}dw\bigg)^{\frac{1}{q_0}}
\\& \quad
\lesssim
\int_{\Gamma}\bigg(\dashint_{B(x,2^{l+2}t)}
\big|z^{\frac{1}{2}}\nabla e^{-zL}b_i\big|^{q_0}dw\bigg)^{\frac{1}{q_0}}\frac{r_{B_i}^{2M}}{|z|^{\frac{1}{2}}(|z|+t^2)^{M+1}}|dz|
\\& \quad
\lesssim
\int_{\Gamma}\bigg(\dashint_{B(x,2^{l+2}t)}
\big|\chi_{2^{j+2}B_i\setminus 2^{j-1}B_i}z^{\frac{1}{2}}\nabla e^{-zL}b_i\bigg|^{q_0}dw\big)^{\frac{1}{q_0}}\frac{r_{B_i}^{2M}}{|z|^{M+\frac{3}{2}}}|dz|
\\& \quad
\lesssim
\int_{\Gamma}\mathcal{M}_{q_0}^{w}\big(\chi_{2^{j+2}B_i\setminus 2^{j-1}B_i}z^{\frac{1}{2}}\nabla e^{-zL}b_i\big)(x)\frac{r_{B_i}^{2M}}{|z|^{M+\frac{3}{2}}}|dz|.
\end{align*}
Applying again  Minkowski's integral inequality, and recalling that $\mathcal{M}_{q_0}^w$ is bounded on $L^{p_2}(w)$ since  $q_0<p_2$,  we get
\begin{align*}
\mathfrak{C}_1
&\lesssim
\int_{\Gamma}
\bigg(
\int_{C_j(B_i)}
\mathcal{M}_{q_0}^w\big(
\chi_{2^{j+2}B_i\setminus 2^{j-1}B_i}z^{\frac{1}{2}}\nabla e^{-zL}b_i\big)^{p_2}
dw\bigg)^{\frac{1}{p_2}}
\frac{r_{B_i}^{2M}}{|z|^{M+\frac{3}{2}}}|dz|
\\
&\lesssim
\int_{\Gamma}\bigg(\int_{2^{j+2}B_i\setminus 2^{j-1}B_i}\big|z^{\frac{1}{2}}\nabla e^{-zL}b_i\big|^{p_2}dw\bigg)^{\frac{1}{p_2}}
\frac{r_{B_i}^{2M}}{|z|^{M+\frac{3}{2}}}|dz|.
\end{align*}
Observe that  $2^{j+2}B_i\setminus 2^{j-1}B_i=\cup_{l=1}^3C_{l+j-2}(B_i)$.
Then by the fact that $z^{\frac{1}{2}}\nabla e^{-zL}\in\mathcal{O}(L^{p_2}(w)-L^{p_2}(w))$,  \eqref{CZ-ptildeexta}, and changing the variable $s$ into $\frac{4^jr_{B_i}^2}{s^2}$, 
\begin{align*}
\mathfrak{C}_1
&\lesssim
w(2^{j+1}B_i)^{\frac{1}{p_2}}
2^{j\theta_1}
\bigg(\dashint_{B_i}|b_i|^{p_2}dw\bigg)^{\frac{1}{p_2}}
\int_{0}^{\infty}\Upsilon\bigg(\frac{2^jr_{B_i}}{s^{\frac{1}{2}}}\bigg)^{\theta_2}
e^{-c\frac{4^jr_{B_i}^2}{s}}
\frac{sr_{B_i}^{2M}}{s^{M+\frac{3}{2}}}\frac{ds}{s}
\\
&\lesssim
\alpha
w(2^{j+1}B_i)^{\frac{1}{p_2}}
2^{-j(2M+1-\theta_1)}
\int_{0}^{\infty}\Upsilon(s)^{\theta_2}
e^{-cs^2}
s^{2M+1}
\frac{ds}{s}
\\ &\lesssim
\alpha
w(2^{j+1}B_i)^{\frac{1}{p_2}}
2^{-j(2M+1-\theta_1)},
\end{align*}
 provided $2M+1>\theta_2$.

We continue by estimating $\mathfrak{C}_2$. To this end, first change the variable $t$ into $t\sqrt{M+1}=:t\theta_M$. Next, for any $x\in C_j(B_i)$, $j\geq 4$, and $t\geq\frac{2^{j-3}r_{B_i}}{2^{l}\theta_M}$,  note that $B_i\subset B(x_{B_i},\theta_M 2^{l}t)=:B_{i}^{l}\subset B(x,\theta_M2^{l+5}t)$ ($x_{B_i}$ denotes the center of $B_i$). Then,
\begin{multline*}
\mathfrak{C}_2
\lesssim
\bigg(\int_{C_j(B_i)}\bigg(\sup_{t\geq\frac{2^{j-3}r_{B_i}}{2^{l}\theta_M}}\dashint_{B(x,\theta_M 2^{l+2}t)}
\big|\nabla e^{-t^2L}\mathcal{T}_{t,r_{B_i}}(\chi_{B_{i}^{l}}b_i)\big|^{q_0}dw\bigg)^{\frac{p_2}{q_0}}dw\bigg)^{\frac{1}{p_2}}
\\
\lesssim
\bigg(\int_{C_j(B_i)}\bigg(\sup_{t\geq\frac{2^{j-3}r_{B_i}}{2^{l}\theta_M}}w(B(x,\theta_M 2^{l+2}t))^{-1}\int_{\R^n}
\big|\nabla e^{-t^2L}\mathcal{T}_{t,r_{B_i}}(\chi_{B_{i}^{l}}b_i)\big|^{q_0}dw\bigg)^{\frac{p_2}{q_0}}dw\bigg)^{\frac{1}{p_2}},
\end{multline*}
where  $\mathcal{T}_{t,r_{B_i}}:=\bigg(e^{-t^2L}-e^{-(t^2+r_{B_i}^2)L}\bigg)^M$.

In the above setting, first recall that $\max\{r_w,\widehat{q}_-(L)\}<q_0<p_2<q_{+}(L)/s_w$, consequently $t\nabla e^{-t^2L}$
is bounded on $L^{q_0}(w)$. Besides, applying the $L^{q_0}(w)$-$L^{q_0}(w)$ off-diagonal estimates that $\mathcal{T}_{t,r_{B_i}}$ satisfies (see \eqref{AB}), and \eqref{doublingcondition} to obtain
\begin{align}\label{termT}
&\bigg(\int_{\R^n}
\big|\nabla e^{-t^2L}\mathcal{T}_{t,r_{B_i}}(\chi_{B_{i}^{l}}b_i)\big|^{q_0}dw\bigg)^{\frac{1}{q_0}}
\\\nonumber
&\qquad
\lesssim
2^{l}r_{B_i}^{-1}\bigg(\int_{\R^n}
\big|\mathcal{T}_{t,r_{B_i}}(\chi_{B_{i}^{l}}b_i)\big|^{q_0}dw\bigg)^{\frac{1}{q_0}}
\\\nonumber
&\qquad
\lesssim 2^{l}
\sum_{N\geq 1}w(C_N(B_i^{l}))^{\frac{1}{q_0}}
\bigg(\dashint_{C_N(B_{i}^{l})}
\bigg|\mathcal{T}_{t,r_{B_i}}\bigg(\chi_{B_{i}^{l}}\frac{b_i}{r_{B_i}}\bigg)\bigg|^{q_0}dw\bigg)^{\frac{1}{q_0}}
\\\nonumber
&\qquad
\lesssim 2^{l\theta}w(B_i^{l})^{\frac{1}{q_0}}
\sum_{N\geq 1}e^{-c4^N}
\bigg(\frac{r_{B_i}^2}{t^2}\bigg)^{M}\bigg(\dashint_{B_{i}^{l}}
\bigg|\frac{b_i}{r_{B_i}}\bigg|^{q_0}dw\bigg)^{\frac{1}{q_0}}
\\\nonumber
&\qquad
\lesssim 2^{l\theta}w(B_i)^{\frac{1}{q_0}}
\alpha\bigg(\frac{r_{B_i}^2}{t^2}\bigg)^{M}
,
\end{align}
where in the last inequality we have used \eqref{CZ-ptildeexta} since $q_0<p_2$.
Consequently, 
\begin{align*}
\mathfrak{C}_2
\lesssim  2^{l\theta}\alpha
\bigg(\int_{C_j(B_i)}\bigg(\sup_{t\geq \frac{2^{j-l-3}r_{B_i}}{\theta_M}}\bigg(\frac{r_{B_i}^2}{t^2}\bigg)^{M}\bigg)^{p_2}dw\bigg)^{\frac{1}{p_2}}
\lesssim
\alpha
w(2^{j+1}B_i)^{\frac{1}{p_2}}2^{-j2M}2^{l(2M+\theta)}
,
\end{align*}
where in the first inequality, we have used that $w(B(x,\theta_M2^{l+2}t))^{-1}w(B_i)\leq C$,  since $B_{i}\subset B(x,\theta_M2^{l+5}t)$ and by \eqref{doublingcondition}.

Collecting the estimates obtained for $\mathfrak{C}_1$ and  $\mathfrak{C}_2$, we conclude that, for $M\in \N$ such that $2M+1>\theta_2$,
\begin{align}\label{term:I_1^ij}
I_{1,l}^{ij}\lesssim \alpha\, w(2^{j+1}B_i)^{\frac{1}{p_2}}  2^{-j(2M-\theta_1)} 2^{l(2M+\theta)}.
\end{align}
This completes the proof of Claim \ref{claim} for $m=1$. 
\end{proof}

\begin{proof}[Proof of Claim \ref{claim}, $m=2$]
Splitting the supremum in $t$, we have
\begin{align*}
I_{2,l}^{ij}&
\lesssim 
\bigg(
\int_{C_j(B_i)}
\sup_{0<t<2^{j-l-3}r_{B_i}}\bigg(\dashint_{B(x,2^{l+2}t)}
\widetilde{\mathrm{G}}_{\hh}\big(\mathcal{B}_{r_{B_i}}b_i\big)^{q_0}dw\bigg)^{\frac{p_2}{q_0}}dw
\bigg)^{\frac{1}{p_2}}
\\
&\qquad+
\bigg(
\int_{C_j(B_i)}
\sup_{t\geq 2^{j-l-3}r_{B_i}}\bigg(\dashint_{B(x,2^{l+2}t)}
\widetilde{\mathrm{G}}_{\hh,t}\big(\mathcal{B}_{r_{B_i}}b_i\big)^{q_0}dw\bigg)^{\frac{p_2}{q_0}}dw
\bigg)^{\frac{1}{p_2}}
\\&=:\,D_1^{ij}+D_2^{ij}.
\end{align*}

\noindent\textbf{Estimate for $D_{1}^{ij}$:}  
We claim that, for some  fixed constants $\widetilde{\theta}_1>0$, $\widetilde{\theta}_2>0$, and for $M$ large enough the following holds
\begin{align}\label{D1ijfinal}
D_{1}^{ij}
\lesssim
 \alpha
w(2^{j+1}B_i)^{\frac{1}{p_2}}
2^{-j\left(2M-\widetilde{\theta_1}-2\widetilde{\theta_2}\right)}.
\end{align}
To  show this, first note that as before, for $0<t<2^{j-l-3}r_{B_i}$ and $x\in C_j(B_i)$,  we have that $B(x,2^{l+2}t)\subset 2^{j+2}B_i\setminus 2^{j-1}B_i$. Next recall that  $\mathcal{M}_{q_0}^w$ is bounded on $L^{p_2}(w)$  since $q_0<p_2$. Hence,
\begin{align*}
D_1^{ij}
&\lesssim
\bigg(\int_{C_j(B_i)} \mathcal{M}_{q_0}^w \big(\chi_{2^{j+2}B_i\setminus 2^{j-1}B_i}
\widetilde{\mathrm{G}}_{\hh}\widehat{b}_i\big)^{p_2}dw \bigg)^{\frac{1}{p_2}}
\\ &\lesssim 
w(2^{j+1}B_i)^{\frac{1}{p_2}} \bigg(
\dashint_{2^{j+2}B_i\setminus 2^{j-1}B_i}
\big|
\widetilde{\mathrm{G}}_{\hh}\widehat{b}_i
\big|^{p_2}dw
\bigg)^{\frac{1}{p_2}}.
\end{align*}
%
Now, note that since $p_2<q_+(L)/s_w$, we can take $\max\{2,p_2\}<q<q_+(L)$ so that $w\in RH_{\big(\frac{q}{p_2}\big)'}$.
Hence by Lemma \ref{lemma:sin-con} and Minkowski's integral inequality, we have
\begin{align}\label{D1ij}
D_1^{ij}
\lesssim
w(2^{j+1}B_i)^{\frac{1}{p_2}}
\bigg(\int_0^{\infty}\bigg(
\dashint_{2^{j+2}B_i\setminus 2^{j-1}B_i}
\big|
r\nabla rLe^{-r^2L}\widehat{b}_i
\big|^{q}dx\bigg)^{\frac{2}{q}}\frac{dr}{r}
\bigg)^{\frac{1}{2}}.
\end{align}
In order to estimate the integral 
in $x$, we use functional calculus  with the same choice of $\phi$ as in the estimate of $\mathfrak{C}_1$. Then
\begin{multline}\label{D1-1}
\bigg(\dashint_{2^{j+2}B_i\setminus 2^{j-1}B_i}\big|r\nabla rLe^{-r^2L}\widehat{b}_i\big|^{q}dx\bigg)^{\frac{1}{q}}
\\
\lesssim
\int_{\Gamma}\bigg(
\dashint_{2^{j+2}B_i\setminus 2^{j-1}B_i}
\big|z^{\frac{1}{2}}\nabla zL e^{-zL} b_i\big|^{q}dx\bigg)^{\frac{1}{q}}
\frac{r^2r_{B_i}^{2M}}{|z|^{\frac{1}{2}}(|z|+r^2)^{M+1}}\frac{|dz|}{|z|}.
\end{multline}
Split the integral in $x$ as follows
\begin{multline}\label{D1-2}
\bigg(
\dashint_{2^{j+2}B_i\setminus 2^{j-1}B_i}
\big|z^{\frac{1}{2}}\nabla zL e^{-zL} b_i\big|^{q}dx\bigg)^{\frac{1}{q}}
\\
\lesssim\sum_{l=1}^{j-3}
\bigg(
\dashint_{2^{j+2}B_i\setminus 2^{j-1}B_i}
\big|z^{\frac{1}{2}}\nabla zL e^{-\frac{z}{2}L}(\chi_{C_l(B_i)} e^{-\frac{z}{2}L}b_i)\big|^{q}dx\bigg)^{\frac{1}{q}}
\\
\quad +\sum_{l\geq j-2}
\bigg(
\dashint_{2^{j+2}B_i\setminus 2^{j-1}B_i}
\big|z^{\frac{1}{2}}\nabla zL e^{-\frac{z}{2}L}(\chi_{C_l(B_i)} e^{-\frac{z}{2}L}b_i)\big|^{q}dx\bigg)^{\frac{1}{q}}
=:\mathfrak{A}+\mathfrak{B}.
\end{multline}
Note now that since
$j\geq 4$, for $1\leq l\leq j-3$ we have that $d(2^{j+2}B_i\setminus 2^{j-1}B_i,C_l(B_i))\geq 2^{j-2}r_{B_i}\geq 2^{l+1}r_{B_i}$.
Then, in that case, first applying the fact that $\sqrt{\tau}\nabla \tau Le^{-\tau L}$ satisfies $L^{q_2}(\R^n)$- $L^{{q}}(\R^n)$ off-diagonal estimates and  split the exponential term (recall that $l\leq j-3$). Next apply Lemma \ref{lemma:sin-con} since $w\in A_{\frac{p_1}{q_2}}$ (see \eqref{p_1choice}) to get
\begin{align*}
\mathfrak{A}&\lesssim |2^{j+1}B_i|^{-\frac{1}{q}}
\sum_{l=1}^{j-3}
\bigg(
\int_{2^{j+2}B_i\setminus 2^{j-1}B_i}
\big|z^{\frac{1}{2}}\nabla zL e^{-\frac{z}{2}L}(\chi_{C_l(B_i)} e^{-\frac{z}{2}L}b_i)\big|^{q}dx\bigg)^{\frac{1}{q}}
\\&
\lesssim |2^{j+1}B_i|^{-\frac{1}{q}}
\sum_{l=1}^{j-3}
\bigg(
\int_{C_l(B_i)}
|e^{-\frac{z}{2}L}b_i|^{q_2}dx\bigg)^{\frac{1}{q_2}}
e^{-c\frac{4^jr_{B_i}^2}{|z|}}
|z|^{-\frac{n}{2}(\frac{1}{q_2}-\frac{1}{q})}
\\&
\lesssim|2^{j+1}B_i|^{-\frac{1}{q}}
\sum_{l=1}^{j-3} |2^{l+1}B_i|^{\frac{1}{q_2}}
\bigg(
\dashint_{C_l(B_i)}
|e^{-\frac{z}{2}L}b_i|^{p_1}dw\bigg)^{\frac{1}{p_1}}
e^{-c\frac{4^jr_{B_i}^2}{|z|}}e^{-c\frac{4^lr_{B_i}^2}{|z|}}
|z|^{-\frac{n}{2}(\frac{1}{q_2}-\frac{1}{q})}.
\end{align*}
Now, since $e^{-\frac{z}{2}L}\in \mathcal{O}(L^{p_2}(w)-L^{p_1}(w))$ and by \eqref{CZ-ptildeexta}
\begin{multline*}
\bigg(
\dashint_{C_l(B_i)}
|e^{-\frac{z}{2}L}b_i|^{p_1}dw\bigg)^{\frac{1}{p_1}}
\lesssim 2^{l\widetilde{\theta_1}}
\Upsilon\bigg(\frac{2^{l+1}r_{B_i}}{|z|^{\frac{1}{2}}}\bigg)^{\widetilde{\theta_2}}\bigg(
\dashint_{B_i}
|b_i|^{p_2}dw\bigg)^{\frac{1}{p_2}}
\\
\lesssim \alpha r_{B_i} 2^{l\widetilde{\theta_1}}
\Upsilon\bigg(\frac{2^{l+1}r_{B_i}}{|z|^{\frac{1}{2}}}\bigg)^{\widetilde{\theta_2}}.
\end{multline*}
Hence,
\begin{align*}
\mathfrak{A}
\lesssim \alpha r_{B_i} |2^{j+1}B_i|^{-\frac{1}{q}}
\sum_{l=1}^{j-3}2^{l\widetilde{\theta_1}}
\Upsilon\bigg(\frac{2^{l+1}r_{B_i}}{|z|^{\frac{1}{2}}}\bigg)^{\widetilde{\theta_2}} |2^{l+1}B_i|^{\frac{1}{q_2}}
e^{-c\frac{4^jr_{B_i}^2}{|z|}}e^{-c\frac{4^lr_{B_i}^2}{|z|}}
|z|^{-\frac{n}{2}(\frac{1}{q_2}-\frac{1}{q})}.
\end{align*}

If we now consider $l\geq j-2$, in this case, we do not have distance between $2^{j+2}B_i\setminus 2^{j-1}B_i$ and $C_l(B_i)$, but we do have it between $C_l(B_i)$ and $B_i$. Indeed, since $l\geq j-2\geq 2$ (recall that $j\geq 4$), we have that $d(C_l(B_i),B_i)> 2^{l-1} r_{B_i}\geq 2^{j-3} r_{B_i}$. Hence, proceeding as in the above computation, we obtain
\begin{align*}
\mathfrak{B}
&\lesssim |2^{j+1}B_i|^{-\frac{1}{q}}\sum_{l\geq j-2}
\bigg(\int_{2^{j+2}B_i\setminus 2^{j-1}B_i}\big|z^{\frac{1}{2}}\nabla z Le^{-\frac{z}{2}L}\big(\chi_{C_l(B_i)}e^{-\frac{z}{2}L}b_i\big)\big|^{q}dy\bigg)^{\frac{1}{q}}
\\
&\lesssim |2^{j+1}B_i|^{-\frac{1}{q}}\sum_{l\geq j-2}
\bigg(\int_{C_l(B_i)}\big|e^{-\frac{z}{2}L}b_i(y)\big|^{q_2}dy\bigg)^{\frac{1}{q_2}}|z|^{-\frac{n}{2}(\frac{1}{q_2}-\frac{1}{q})}
\\
&\lesssim|2^{j+1}B_i|^{-\frac{1}{q}} \sum_{l\geq j-2}(2^{l} r_{B_i})^{\frac{n}{q_2}}
\bigg(\dashint_{C_l(B_i)}\big|e^{-\frac{z}{2}L}b_i(y)\big|^{p_1}dw\bigg)^{\frac{1}{p_1}}|z|^{-\frac{n}{2}(\frac{1}{q_2}-\frac{1}{q})}
\\
&
\lesssim |2^{j+1}B_i|^{-\frac{1}{q}}\sum_{l\geq j-2}2^{l\widetilde{\theta_1}}
(2^{l} r_{B_i})^{\frac{n}{q_2}}\Upsilon\bigg(\frac{2^{l+1}r_{B_i}}{|z|^{\frac{1}{2}}}\bigg)^{\widetilde{\theta_2}}e^{-c\frac{4^lr_{B_i}^2}{|z|}}\bigg(\dashint_{B_i}|b_i|^{p_2} dw\bigg)^{\frac{1}{p_2}}|z|^{-\frac{n}{2}(\frac{1}{q_2}-\frac{1}{q})}
\\
&
\lesssim \alpha r_{B_i}|2^{j+1}B_i|^{-\frac{1}{q}}\sum_{l\geq j-2}2^{l\widetilde{\theta_1}}
(2^{l} r_{B_i})^{\frac{n}{q_2}}\Upsilon\bigg(\frac{2^{l+1}r_{B_i}}{|z|^{\frac{1}{2}}}\bigg)^{\widetilde{\theta_2}}e^{-c\frac{4^lr_{B_i}^2}{|z|}}e^{-c\frac{4^jr_{B_i}^2}{|z|}}|z|^{-\frac{n}{2}(\frac{1}{q_2}-\frac{1}{q})}.
\end{align*}
Thus, from the estimates of $\mathfrak{A}$ and $\mathfrak{B}$, in view of \eqref{D1ij}, \eqref{D1-1}, and \eqref{D1-2}, we have
\begin{align}\label{D1-3}
D_1^{ij}
\lesssim& \alpha r_{B_i}|2^{j+1}B_i|^{-\frac{1}{q}}
w(2^{j+1}B_i)^{\frac{1}{p_2}}\sum_{l\geq 1}2^{l\widetilde{\theta_1}}
(2^{l} r_{B_i})^{\frac{n}{q_2}}
\\\nonumber&
\bigg(\int_0^{\infty}\bigg(\int_{\Gamma}\Upsilon\bigg(\frac{2^{l+1}r_{B_i}}{|z|^{\frac{1}{2}}}\bigg)^{\widetilde{\theta_2}}e^{-c\frac{4^lr_{B_i}^2}{|z|}}e^{-c\frac{4^jr_{B_i}^2}{|z|}}|z|^{-\frac{n}{2}(\frac{1}{q_2}-\frac{1}{q})}
\frac{|z|^{-\frac{1}{2}}r^2r_{B_i}^{2M}}{(|z|+r^2)^{M+1}}\frac{|dz|}{|z|}\bigg)^{2}\frac{dr}{r}
\bigg)^{\frac{1}{2}}.
\end{align}
Doing the  change of variables $s$ into $4^jr_{B_i}^2/s^2$, we obtain
\begin{align}\label{D1-4}
&
\int_0^{\infty}\Upsilon\bigg(\frac{2^{l+1}r_{B_i}}{s^{\frac{1}{2}}}\bigg)^{\widetilde{\theta_2}}
e^{-c\frac{4^lr_{B_i}^2}{s}}
e^{-c\frac{4^jr_{B_i}^2}{s}}
s^{-\frac{n}{2}(\frac{1}{q_2}-\frac{1}{q})-\frac{1}{2}}\frac{r^2r_{B_i}^{2M}}{(s+r^2)^{M+1}}\frac{ds}{s}
\\\nonumber&\quad
\lesssim (2^jr_{B_i})^{-n\bigg(\frac{1}{q_2}-\frac{1}{q}\bigg)-1}
\int_0^{\infty}\Upsilon\bigg(\frac{2^ls}{2^j}\bigg)^{\widetilde{\theta_2}}
e^{-c\frac{4^{l}s^2}{4^j}}
e^{-s^2}
s^{n(\frac{1}{q_2}-\frac{1}{q})+1}\frac{r^2r_{B_i}^{2M}}{(4^jr_{B_i}^2/s^2+r^2)^{M+1}}\frac{ds}{s}.
\end{align}
Besides, changing the variable $r$ into $2^j r_{B_i}r$, we have
\begin{multline}\label{D1-5}
\bigg(\int_0^{\infty}
\bigg(\int_{0}^{\infty}\Upsilon\bigg(\frac{2^{l}s}{2^{j}}\bigg)^{\widetilde{\theta_2}}s^{n(\frac{1}{q_2}-\frac{1}{q})+1}e^{-cs^2}e^{-c\frac{4^ls^2}{4^j}}
r^2\frac{ r_{B_i}^{2M}}{(4^{j} r_{B_i}^2/s^2+r^2)^{M+1}}\frac{ds}{s}\bigg)^{2}\frac{dr}{r}\bigg)^{\frac{1}{2}}
\\
\lesssim 2^{-j2M}
\bigg(\int_0^{\infty}
r^4\bigg(\int_{0}^{\infty}\Upsilon\bigg(\frac{2^{l}s}{2^{j}}\bigg)^{\widetilde{\theta_2}}
\frac{s^{n(\frac{1}{q_2}-\frac{1}{q})+1}e^{-cs^2}e^{-c\frac{4^ls^2}{4^j}}}{(1/s^2+r^2)^{M+1}}\frac{ds}{s}\bigg)^{2}\frac{dr}{r}\bigg)^{\frac{1}{2}}
.
\end{multline}
In order to bound the above integral, take   $\widetilde{M}=\frac{1}{2}(\widetilde{\theta_2}+\widetilde{\theta_1}+\frac{n}{q_2}+1)$ and $M\geq 2$ so that $2M+n(\frac{1}{q_2}-\frac{1}{q})-2\widetilde{M}-\widetilde{\theta_2}-1>0
$. Then, 
\begin{align*}
&\bigg(\int_0^{1}
r^4\bigg(\int_{0}^{\infty}\Upsilon\bigg(\frac{2^{l}s}{2^{j}}\bigg)^{\widetilde{\theta_2}}s^{n(\frac{1}{q_2}-\frac{1}{q})+1}e^{-cs^2}e^{-c\frac{4^ls^2}{4^j}}
\frac{1}{(1/s^2+r^2)^{M+1}}\frac{ds}{s}\bigg)^{2}\frac{dr}{r}\bigg)^{\frac{1}{2}}
\\&\quad
\lesssim 2^{-l(2\widetilde{M}-\widetilde{\theta_2})}2^{j(2\widetilde{M}+\widetilde{\theta_2})}\bigg(
\bigg(\int_0^{1}
r^4\bigg(\int_{0}^{1}s^{2M+3+n(\frac{1}{q_2}-\frac{1}{q})-2\widetilde{M}-\widetilde{\theta_2}}
\frac{ds}{s}\bigg)^{2}\frac{dr}{r}\bigg)^{\frac{1}{2}}\bigg.
\\& \qquad\qquad\qquad\qquad
\bigg.
+\bigg(\int_0^{1}
r^4\bigg(\int_{1}^{\infty}s^{2M+3+n(\frac{1}{q_2}-\frac{1}{q})-2\widetilde{M}+\widetilde{\theta_2}}e^{-cs^2}
\frac{ds}{s}\bigg)^{2}\frac{dr}{r}\bigg)^{\frac{1}{2}}\bigg)
\\&\quad
\lesssim 2^{-l(2\widetilde{M}-\widetilde{\theta_2})}2^{j(2\widetilde{M}+\widetilde{\theta_2})}.
\end{align*}
%
%
And, since $=(s^{-2}+r^2)^{-M-1}=(s^{-2}+r^2)^{-M+1}(s^{-2}+r^2)^{-2}\leq s^{2(M-1)}r^{
-4}$,
\begin{align*}
&\bigg(\int_1^{\infty}
r^4\bigg(\int_{0}^{\infty}\Upsilon\bigg(\frac{2^{l}s}{2^{j}}\bigg)^{\widetilde{\theta_2}}s^{n(\frac{1}{q_2}-\frac{1}{q})+1}e^{-cs^2}e^{-c\frac{4^ls^2}{4^j}}
\frac{1}{(1/s^2+r^2)^{M+1}}\frac{ds}{s}\bigg)^{2}\frac{dr}{r}\bigg)^{\frac{1}{2}}
\\&\quad
\lesssim
\bigg(\int_1^{\infty}
r^{-4}\bigg(\int_{0}^{\infty}\Upsilon\bigg(\frac{2^{l}s}{2^{j}}\bigg)^{\widetilde{\theta_2}}s^{2M+n(\frac{1}{q_2}-\frac{1}{q})-1}e^{-cs^2}e^{-c\frac{4^ls^2}{4^j}}
\frac{ds}{s}\bigg)^{2}\frac{dr}{r}\bigg)^{\frac{1}{2}}
\\&\quad
\lesssim 2^{-l(2\widetilde{M}-\widetilde{\theta_2})}
2^{j(2\widetilde{M}+\widetilde{\theta_2})}.
\end{align*}
%
%
Hence, by our choice of $M$ and $\widetilde{M}$
\begin{multline*}
D_1^{ij}
\lesssim
 \alpha
w(2^{j+1}B_i)^{\frac{1}{p_2}}
2^{-j(2M+1+\frac{n}{q_2}-2\widetilde{M}-\widetilde{\theta_2})}
\sum_{l\geq 1}
2^{-l(2\widetilde{M}-\widetilde{\theta_2}-\widetilde{\theta_1}-\frac{n}{q_2})}
\\
\lesssim
 \alpha
w(2^{j+1}B_i)^{\frac{1}{p_2}}
2^{-j(2M-\widetilde{\theta_1}-2\widetilde{\theta_2})}.
\end{multline*}


%
\noindent\textbf{Estimate for $D_{2}^{ij}$:}  We claim that, for some  fixed constant $\widetilde{\theta}>0$ and for $M$ large enough the following holds
\begin{align}\label{D2ij}
D_2^{ij}
\lesssim 
2^{l\left(2M+\frac{n}{q_2}+\widetilde{\theta}\right)}\alpha w(2^{j+1}B_i)^{\frac{1}{p_2}} 2^{-2jM}.
\end{align}
For any $t\geq 2^{j-l-3}r_{B_i}$ and $f\in L^2(\R^n)$, we have that
$$
\widetilde{\mathrm{G}}_{\hh,t}f(x)=\left(\int_{\frac{t}{2}}^{\infty}|r\nabla rLe^{-r^2L}f(x)|^2\frac{dr}{r}\right)^{\frac{1}{2}}\leq 
\left(\int_{2^{j-l-4}r_{B_i}}^{\infty}|r\nabla rLe^{-r^2L}f(x)|^2\frac{dr}{r}\right)^{\frac{1}{2}}.
$$
Take $q_1$ such that $q<q_1<q_+(L)$, and recall that $q_0<p_2$. Consequently, we  get that $w\in RH_{\left(\frac{q}{p_2}\right)'}\subset RH_{\left(\frac{q_1}{p_2}\right)'}$ and $w\in RH_{\left(\frac{q}{p_2}\right)'}\subset RH_{\left(\frac{q}{q_0}\right)'}$. Hence, by
Lemma \ref{lemma:sin-con} we have that $\Mcal_{q_0}^wh\leq \Mcal_{q}h$ and  the third inequality below. 
Next, we use
the boundedness of $\Mcal_{q}$ on $L^{q_1}(\R^n)$, since $q_1>q$,
 and Minkowski's integral inequality  to obtain
\begin{align}\label{D2ijplug}
&D_2^{ij}
\lesssim
\bigg(
\int_{C_j(B_i)}
\mathcal{M}^w_{q_0}
\big({\widetilde{\mathrm{G}}_{\hh,2^{j-l-3}r_{B_i}}
(\,\widehat{b}_i\,)}\big)^{p_2}dw
\bigg)^{\frac{1}{p_2}}
\\\nonumber
&
\lesssim
\bigg(
\int_{C_j(B_i)}
\mathcal{M}_{q}
\big({\widetilde{\mathrm{G}}_{\hh,2^{j-l-3}r_{B_i}}
(\,\widehat{b}_i\,)}\big)^{p_2}dw
\bigg)^{\frac{1}{p_2}}
\\\nonumber
&
\lesssim w(2^{j+1}B_i)^{\frac{1}{p_2}}
\bigg(
\dashint_{C_j(B_i)}
\mathcal{M}_{q}
\big(\widetilde{\mathrm{G}}_{\hh,2^{j-l-3}r_{B_i}}
(\,\widehat{b}_i\,)\big)^{q_1}dx
\bigg)^{\frac{1}{q_1}}
\\\nonumber
&
\lesssim w(2^{j+1}B_i)^{\frac{1}{p_2}}|2^{j+1}B_i|^{-\frac{1}{q_1}}
\bigg(
\int_{\R^n}\bigg(
\int_{2^{j-l-4}r_{B_i}}^{\infty}
\big|
r\nabla rLe^{-r^2L} \widehat{b}_i
\big|^{2}\frac{dr}{r}\bigg)^{\frac{q_1}{2}}dx\bigg)^{\frac{1}{q_1}}
\\\nonumber
&
\lesssim
w(2^{j+1}B_i)^{\frac{1}{p_2}}|2^{j+1}B_i|^{-\frac{1}{q_1}}
\bigg(
\int_{\frac{2^{j-l-4}r_{B_i}}{\theta_M}}^{\infty}\bigg(\int_{\R^n}
\big|r\nabla rLe^{-r^2L}\mathcal{T}_{r,r_{B_i}}(\chi_{B_{i}^{l}} b_i)
\big|^{q_1}dx\bigg)^{\frac{2}{q_1}}\frac{dr}{r}\bigg)^{\frac{1}{2}}
,
\end{align}
where in the last inequality we have changed the variable $r$ into $r\theta_M:=r\sqrt{M+1}$, used that  $B_i\subset B(x_{B_i}, \theta_M 2^{l}r)=:B_{i}^{l}$, for $r>\frac{2^{j-l-4}r_{B_i}}{\theta_M}$ and $j\geq 4$ ($x_{B_i}$ denotes the center of $B_i$), and we recall that  
$\mathcal{T}_{r,r_{B_i}}:=(e^{-r^2L}-e^{-(r^2+r_{B_i}^2)L})^M$. 

By  the $L^{q_2}(\R^n)$-$L^{q_1}(\R^n)$ off-diagonal estimates of $\tau\nabla \tau^2L e^{-\tau^2L}$, and  
\eqref{boundednesstsr} (with $w\equiv 1$ and $p=q_2$), (recall the choice of $q_2$ in \eqref{p_1choice} and that $2<q_1<q_+(L)$)
\begin{align*}
&\bigg(\int_{\R^n}\big|r\nabla rLe^{-r^2L}\mathcal{T}_{r,r_{B_i}}\big(\chi_{B_{i}^{l}}b_i)\big|^{q_1}dx\bigg)^{\frac{1}{q_1}}
\\&\quad
\lesssim 
2^{l}r_{B_i}^{-1}r^{-n(\frac{1}{q_2}-\frac{1}{q_1})}
\bigg(\int_{\R^n}\big|\mathcal{T}_{r,r_{B_i}}e^{-\frac{r^2}{2}L}(\chi_{B_{i}^{l}}b_i)\big|^{q_2}dx\bigg)^{\frac{1}{q_2}}
\\&\quad
\lesssim 
2^{l}
r_{B_i}^{-1}
\bigg(\frac{r_{B_i}^2}{r^2}\bigg)^{M}r^{-n(\frac{1}{q_2}-\frac{1}{q_1})}
\bigg(\int_{\R^n}\big|e^{-\frac{r^2}{2}L}(\chi_{B_{i}^{l}}b_i)\big|^{q_2}dx\bigg)^{\frac{1}{q_2}}.
\end{align*}
Since $w\in A_{\frac{p_1}{q_2}}$ (see \eqref{p_1choice}), by Lemma \ref{lemma:sin-con}, and the $L^{p_2}(w)$-$L^{p_1}(w)$ off-diagonal estimates on balls satisfied by $e^{-\tau L}$,
we have that
\begin{align*}
\bigg(\int_{\R^n}\big|e^{-\frac{r^2}{2}L}(\chi_{B_{i}^{l}}b_i)\big|^{q_2}dx\bigg)^{\frac{1}{q_2}}
&\lesssim
\sum_{N\geq 1}|2^{N+1}B_{i}^{l}|^{\frac{1}{q_2}}
\bigg(\dashint_{C_N(B_{i}^{l})}\big|e^{-\frac{r^2}{2}L}({\chi_{B_{i}^{l}}b_i})\big|^{q_2}dx\bigg)^{\frac{1}{q_2}}
\\&
\lesssim
\sum_{N\geq 1}|2^{N+1}B_{i}^{l}|^{\frac{1}{q_2}}
\bigg(\dashint_{C_N(B_{i}^{l})}\big|{e^{-\frac{r^2}{2}L}(\chi_{B_{i}^{l}}b_i)}\big|^{p_1}dw\bigg)^{\frac{1}{p_1}}
\\&
\lesssim 2^{l\widetilde{\theta}}
|B_{i}^{l}|^{\frac{1}{q_2}}
\bigg(\dashint_{B_{i}^{l}}\abs{b_i}^{p_2}dw\bigg)^{\frac{1}{p_2}}
\\&
\lesssim \alpha r_{B_i}
2^{l\widetilde{\theta}}|B_i^{l}|^{\frac{1}{q_2}},
\end{align*}
where in the last inequality we have used that  
for $r>\frac{2^{j-l-4}r_{B_i}}{\theta_M}$ and $j\geq 4$,  $B_i\subset B_{i}^{l}$.
Plugging this into \eqref{D2ijplug} leads  to
\begin{align*}
D_2^{ij}
& \lesssim 
2^{l(\widetilde{\theta}+\frac{n}{q_2})}\alpha 
w(2^{j+1}B_i)^{\frac{1}{p_2}}|2^{j+1}B_i|^{-\frac{1}{q_1}}
\bigg(
\int_{\frac{2^{j-l-5}r_{B_i}}{\theta_M}}^{\infty}
\bigg(\frac{r_{B_i}^2}{r^2}\bigg)^{2M}r^{\frac{2n}{q_1}}
\frac{dr}{r}\bigg)^{\frac{1}{2}}
\\ &\lesssim 
2^{l\left(\widetilde{\theta}+\frac{n}{q_2}+M\right)}\alpha w(2^{j+1}B_i)^{\frac{1}{p_2}} 2^{-j2M},
\end{align*}
 provided $2M>\frac{n}{q_1}$.

Gather  \eqref{D1ijfinal} and \eqref{D2ij}, then for $M\in \N$ such that $2M>\max\{2\widetilde\theta_2+\widetilde{\theta}_1+2+n/q,n/q_1\}$,
\begin{align*}
I_{2,l}^{ij}\lesssim 2^{l\widetilde{\theta}}\alpha w(2^{j+1}B_i)^{\frac{1}{p_2}}2^{-j(2M-2\widetilde\theta_2-\widetilde{\theta}_1)}.
\end{align*}
This completes the proof of Claim \ref{claim} in the case $m=2$.
\end{proof}

\bibliographystyle{acm}

\end{document}